\documentclass[12pt]{amsart}

\usepackage{amsmath, amsthm, amssymb, fullpage, amscd}

\usepackage[english]{babel}
\usepackage{eucal}
\usepackage{lineno}

\usepackage{thmtools}
\usepackage{float} 
\usepackage{thm-restate}
\usepackage{tikz-cd, pgf, tikz}
\usepackage{geometry}
\usepackage[colorlinks, bookmarks=true, backref=page]{hyperref}
\hypersetup{
	colorlinks,
	citecolor=blue,
	linkcolor=blue,
	urlcolor=blue}
\usepackage{xcolor}
\usepackage{mathtools, mathrsfs}
\usepackage{graphicx}
\usepackage{caption}
\usepackage{subcaption}
\usepackage{bm, bbm}
\usepackage{comment}
\usepackage{array}
\usepackage{latexsym}
\usepackage{stmaryrd}
\usepackage{commath,enumitem,chngcntr}
\usepackage[titletoc,toc,title]{appendix}
\usepackage{pst-node}
\mathtoolsset{showonlyrefs}
\usepackage{overpic}

\usetikzlibrary{arrows,calc,positioning}

\usepackage{natbib}
\usepackage{overpic}

\newcommand{\dmo}{\DeclareMathOperator}
\dmo{\Id}{Id}
\dmo{\pt}{pt}
\dmo{\ab}{ab}
\dmo{\GL}{GL}
\dmo{\SO}{SO}
\dmo{\SU}{SU}
\dmo{\SL}{SL}
\dmo{\PSL}{PSL}
\dmo{\im}{Im}
\dmo{\re}{Re}
\dmo{\Ad}{Ad}
\dmo{\ad}{ad}
\dmo{\tr}{tr}
\dmo{\Log}{Log}
\dmo{\Homeo}{Homeo}
\dmo{\Vol}{Vol}
\dmo{\supp}{supp}
\dmo{\Isom}{Isom}
\dmo{\Hom}{Hom}
\dmo{\Ext}{Ext}
\dmo{\Tor}{Tor}
\dmo{\PP}{\mathbf{P}}
\dmo{\Area}{Area}
\dmo{\Length}{Length}
\dmo{\inte}{int}
\dmo{\sech}{sech}
\dmo{\diag}{diag}

\newcommand{\C}{\mathbb C}
\newcommand{\R}{\mathbb R}

\newcommand{\Z}{\mathbb Z}

\newcommand{\T}{\mathcal T}
\newcommand{\D}{\mathbb{D}}

\renewcommand{\H}{\mathbb H}
\newcommand{\HH}{\mathcal H}

\renewcommand{\tr}{\mathrm {tr}}
\newcommand{\lan}{\langle}
\newcommand{\ran}{\rangle}

\newcommand{\from}{\colon}

\newtheorem{theorem}{Theorem}[section]
\newtheorem{corollary}[theorem]{Corollary}
\newtheorem{proposition}[theorem]{Proposition}
\newtheorem{lemma}[theorem]{Lemma}

\newtheorem{problem}[theorem]{Problem}

\theoremstyle{definition}
\newtheorem{definition}[theorem]{Definition}

\theoremstyle{remark}
\newtheorem{remark}[theorem]{Remark}

\usepackage{lipsum}
\let\OLDthebibliography\thebibliography
\renewcommand\thebibliography[1]{
	\OLDthebibliography{#1}
	\setlength{\parskip}{2pt}
	\setlength{\itemsep}{0pt plus 0.3ex}
}

\numberwithin{equation}{section}

\title{A Rigidity Theorem for Convex Sets in Hyperbolic 3-Space}
\author{Feng Luo}
\address{Department of Mathematics, Rutgers University--New Brunswick, Piscataway, NJ 08854}
\email{fluo@math.rutgers.edu  }

\author{Yanwen Luo}
\address{Department of Mathematics, Oklahoma State University,
Stillwater, OK 74078}
\email{yanwen.luo@okstate.edu}

\author{Zhenghao Rao}
\address{Department of Mathematics, Rutgers University--New Brunswick, Piscataway, NJ 08854}
\email{zhenghao.rao@rutgers.edu}

\date{}

\begin{document}

\begin{abstract} 

Pogorelov's rigidity theorem states that a compact convex body in the hyperbolic 3-space is determined up to isometry by the intrinsic path metric on its boundary.  The main result of this paper addresses a rigidity problem for non-compact closed convex 3-dimensional subsets in hyperbolic 3-space. We show that the intrinsic path metric on the boundary determines a closed convex set up to isometry, provided that the set of limit points of the convex set at infinity of the hyperbolic 3-space has vanishing 1-dimensional Hausdorff measure, i.e., 
zero length.  Furthermore, this zero-length condition is optimal.  This can be considered as an analogue of the Painlev\'e removability theorem in complex analysis, which states that 
sets of zero length are removable for bounded holomorphic functions.   
As a corollary, we show that if the underlying complex structure of a connected polyhedral surface is of parabolic type, then it is discrete conformal, unique up to scaling,  to a complete flat surface marked with a discrete subset.    The proof uses Pogorelov's rigidity theorem for compact convex bodies in $\R^3$, the Pogorelov map, and the Tabor--Tabor theorem on the extension of locally convex functions.

\end{abstract}

\maketitle

\section{Introduction}\label{sec:intro}

\noindent {\bf Background and main result.} Cauchy's rigidity theorem of convex polyhedra, proved in 1813, marks the birth of global rigidity theory in geometry.  It states that if $P$ and $Q$ are two compact convex polyhedra in the Euclidean 3-space $\mathbb{R}^3$ whose boundaries are isometric via an isometry preserving the combinatorial structure, then the two polyhedra differ by a rigid motion of the 3-space. 
Alexandrov~\cite{Alexandroff, Alexandrov} later strengthened Cauchy’s result by removing the requirement that the boundary isometry preserves the combinatorial structure, thereby establishing rigidity under purely metric assumptions. One of the most remarkable results in the theory of rigidity of convex bodies is Pogorelov’s rigidity theorem~\cite{Pogorelov49}, which extends rigidity phenomena far beyond the polyhedral setting.

\begin{theorem}[Pogorelov] \label{pogo1} If $P$ and $Q$ are two compact convex bodies in the 3-dimensional Euclidean space $\R^3$ whose boundaries are isometric with respect to the intrinsic path metrics, then any isometry between the boundaries of $P$ and $Q$ extends to an isometry of $\R^3$. 
\end{theorem}

To extend the result to compact convex bodies in the spherical 3-space $\mathbb S^3$ and the hyperbolic 3-space $\H^3$, 
Pogorelov~\cite{Pogorelov} introduced what is now known as the \emph{Pogorelov map} and used it to prove that Theorem \ref{pogo1} remains valid in both settings.
The purpose of this paper is to prove the following theorem, which is an extension of Pogorelov's theorem to non-compact convex sets in the hyperbolic 3-space.

\begin{theorem} \label{main2} Suppose $X_1$ and $X_2$ are two closed, non-compact convex sets of dimension at least two in the Poincar\'e model $\H_P^3$ of the hyperbolic 3-space such that their boundaries, in the intrinsic path metrics, are isometric and the 1-dimensional Hausdorff measures of the intersections of the closures of $X_i$ in $\R^3$ with $\partial \H_P^3$ are zero. Then any isometry between the boundaries of $X_1$ and $X_2$ extends to an isometry of $\H^3_P$. Moreover, the assumption of the vanishing 1-dimensional Hausdorff measure is sharp.
\end{theorem}
We remark that if one of the $X_i$ is 2-dimensional in the above theorem, the ``boundary of $X_i$" is defined to be the metric double of $X_i$ in the sense of Alexandrov~\cite{BBI}. 
Here,  the metric double of  $X_i$ is obtained by gluing two isometric copies of $X_i$ along their 1-dimensional boundary. The resulting space is a closed Alexandrov surface of non-negative curvature. 
See Section \ref{subsec:2d-2d} for details.   Theorem \ref{main2} reduces to a result of I. Rivin~\cite{Rivin} when $X_i$ are the hyperbolic convex hulls of finite sets of points contained in $\partial \H^3$.

There has been much work on extending Pogorelov's theorem to non-compact 3-dimensional convex sets in hyperbolic 3-space. 
See \cite{ Rivin, Schlenker2001, Schlenker2002,  BO04, Schlenker06, Fillastre, FI09} and others. Rigidity of non-compact convex sets in $\H^3$  arises naturally in our investigation of discrete conformal geometry of polyhedral surfaces. Indeed, the discrete Schwarz lemma and discrete Liouville's theorem can be translated to special types of rigidity of non-compact convex sets in $\H^3$. 
However, the naive extension of Pogorelov's theorem to closed 3-dimensional convex sets in $\H^3$ is false. Here is a counterexample constructed by William Thurston~\cite{Thurston}. Take two Jordan curves $J_1$ and $J_2$ in the boundary sphere at infinity $\partial \H^3$ of the hyperbolic 3-space such that $J_i$ are not round circles,  and let $X_i$ be the hyperbolic convex hull of $J_i$ in $\H^3$. Thurston proved that the boundaries of $X_1$ and $X_2$ are isometric. Indeed, $\partial X_i$ is isometric to the disjoint union of two copies of the hyperbolic plane $\H^2$. Now $X_1$ and $X_2$ differ by an isometry of $\H^3$ if and only if $J_1$ and $J_2$ differ by a M\"obius transformation. Thus, if one takes $J_1$ to be a square and $J_2$ to be a non-square rectangle, then $X_1$ and $X_2$ produce the counterexample. It shows that the vanishing 1-dimensional Hausdorff measure assumption is essential.

For a  set $X \subset \partial \H^3$ containing at least three points, let $C(X)$ be the hyperbolic convex hull of $X$ in $\H^3$. The following corollary is a straightforward consequence of Theorem \ref{main2}.

\begin{corollary} \label{cor3} If $X$ and $Y$ are two compact subsets of $\partial \H^3$ of zero 1-dimensional Hausdorff measure such that $\partial C(X)$ is isometric to $\partial C(Y)$, then $X$ and $Y$ differ by a M\"obius transformation.
\end{corollary}

This corollary can be considered as a parallel result to the well-known Painlev\'e's removability theorem in complex analysis~\cite{Tolsa, Younsi}. 
Namely, if $X \subset \C$ is a compact set of vanishing 1-dimensional Hausdorff measure, and $X$ is contained in an open set $U$ in the plane, then any bounded holomorphic function $f$ defined on $U-X$ extends to a holomorphic function on $U$.  More specifically, Painlev\'e's theorem implies that if $X$ and $Y$ are compact zero-length sets in $\C$ such that $\C-X$ is biholomorphic to $\C-Y$, then $X$ and $Y$ differ by a M\"obius transformation. 

\bigskip
\noindent{\bf Applications to discrete conformal geometry.}
Taking hyperbolic convex hulls establishes a relationship between complex analysis and convex hyperbolic surfaces.
Given an open set $\Omega$ in the complex plane $\C$, considered to lie on the boundary at infinity of the upper-half-space model $\H^3_U$ of the hyperbolic 3-space, one considers the Poincar\'e metric and holomorphic functions on $\Omega$ and also the boundary of the convex hull $\partial C(X)$ where $X =\mathbb{C} \cup \{\infty\}-\Omega : = \Omega^c$. 
Thurston~\cite{Thurston, EM87}
proved that the convex surface $\partial C(X) \subset \H^3_U$, usually called Thurston's dome over $\Omega$, is complete hyperbolic with respect to its intrinsic path metric. Topologically, Thurston's dome $\partial C(X)$ is homeomorphic to the open set $\Omega$ by the vertical projection map $(x_1, x_2, x_3) \mapsto (x_1, x_2)$. If we take $\Omega$ to be a simply connected domain such that $\Omega^c$ contains more than one point, then by the Riemann mapping theorem, there exists a bi-holomorphic map $R: \Omega \to \D$, the unit open disk. On the other hand, Thurston's theorem implies there exists an isometry $T: \partial C(\Omega^c) \to \partial C(\D^c)$. See Figure \ref{tst1} for an illustration.  It shows that Thurston's isometry $T$ is a hyperbolic geometric counterpart of the Riemann mapping $R$. For more details about Thurston's dome, see \cite[Chapter 8]{BishopRMT}, \cite{EM87} and others. 

\begin{figure}[htbp]
  \centering
  \begin{minipage}[t]{0.45\textwidth}
    \centering
    \includegraphics[width=\textwidth]{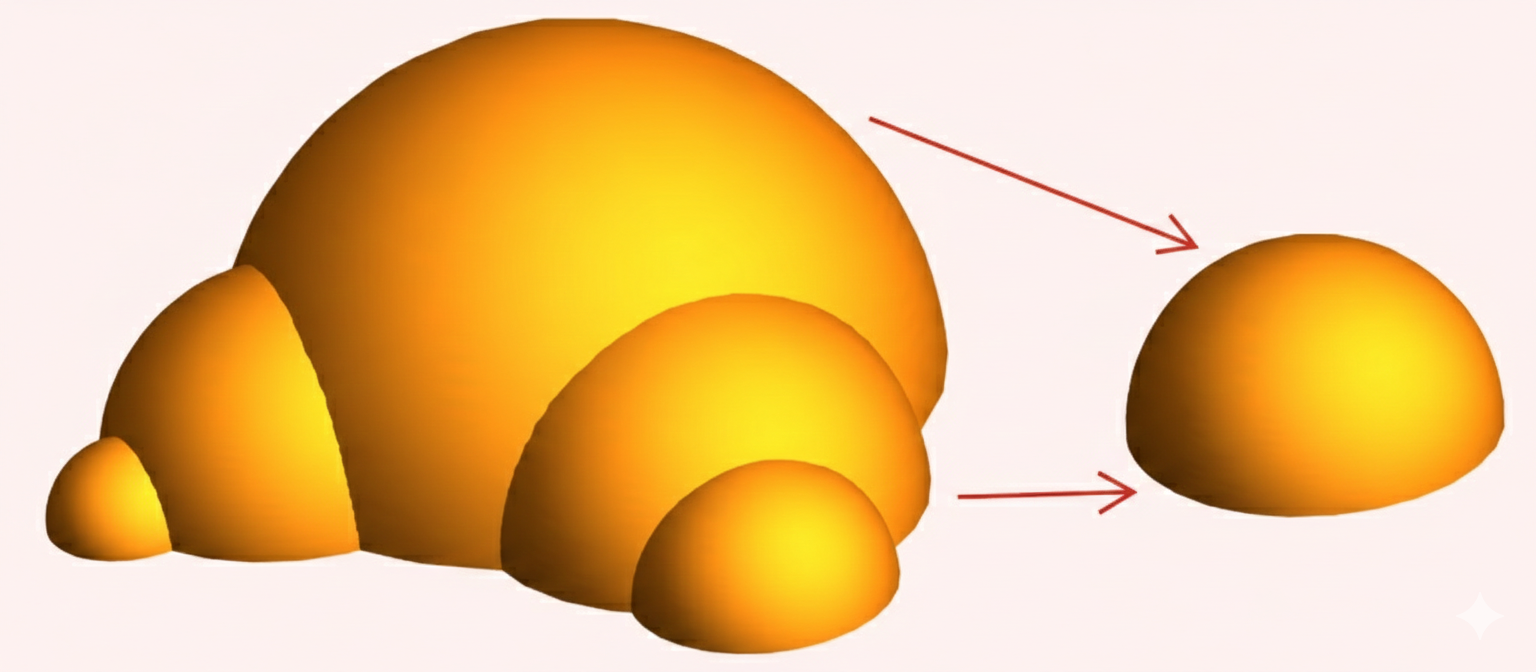}
    \caption{Thurston's isometry $T$ between domes and the Riemann mapping $R$ between the domains below domes.}
     \label{tst1}
  \end{minipage} 
  \hspace{0.5cm}
  \begin{minipage}[t]{0.5\textwidth}
    \centering
    \includegraphics[width=\textwidth]{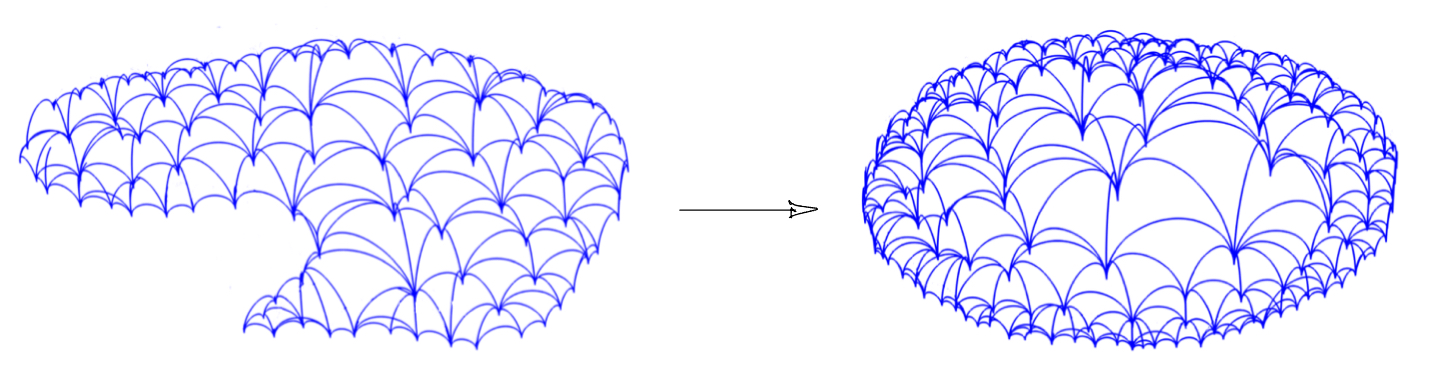}
    \caption{The discrete Riemann mapping $f:V \to W$ is induced by the isometric map between the domes over $\Omega-V$ and $\D-W$.}
          \label{drm3}
  \end{minipage} 
 \end{figure}

The motivation for us to work on Theorem \ref{main2} comes from computational conformal geometry and discrete conformal geometry of polyhedral surfaces. The goal is to discretize conformal maps and Riemann surfaces and to produce efficient algorithms to compute them.  Furthermore, the discrete counterparts should have their own intrinsic geometric structures, i.e., structure-preserving discretizations. See \cite{Luo, BPS, GLSW18a, GLSW18b, LL26, bishop3} and others for more details.  There are two popular schemes to discretize Riemann surfaces based on the observation that a linear conformal transformation of the plane sends circles to circles and preserves the length-cross-ratio.   More precisely, the length-cross-ratio of a quadrilateral $\square v_1v_2v_3v_4$ with respect to the diagonal $v_1v_3$ is defined to be $\frac{|v_1-v_2||v_3-v_4|}{|v_2-v_3||v_4-v_1|}$, i.e. the absolute value of the cross-ratio $(v_1, v_3, v_2, v_4)$. See Figure \ref{ds19}.  If one focuses on discretizing the circle-preserving properties of the conformal maps, one obtains Thurston's theory of circle packing. The length-cross-ratio-preserving property leads to the convex-hull (or the vertex-scaling) discretization of Riemann surfaces. 

Here are two examples of discretization of conformal maps based on the convex-hull construction. The first example is a discretization of the Riemann mapping theorem.  Let $\Omega$ be a bounded simply connected Jordan domain in the plane and $V$ be a discrete subset in $\Omega$ such that the limit points of $V$ in $\R^2$ are equal to  $\partial \Omega$. We consider $V$ to be a discretization of $\Omega$. The convex-hull discrete conformal class of $V$, more precisely $(\Omega, V)$, is defined to be the isometry class of Thurston's dome $\partial C(V \cup \Omega^c)$. In \cite{LW24},  it is proved that there exists a discrete subset $W$ of $\D$ such that $\partial C(V \cup \Omega^c)$ is isometric to $\partial C(W \cup \D^c)$. The pair $(\D, W)$ is considered to be the canonical model of $(\Omega, V)$ in its discrete conformal class.  The isometry map $f: \partial C(V \cup \Omega^c) \to \partial C(W \cup \D^c)$, or simply $f: V \to W$,  is defined to be the \it discrete Riemann mapping\rm. See Figure \ref{drm3} for an illustration.

\begin{figure}[htbp]
  \centering
 \begin{minipage}[t]{0.48\textwidth}
    \centering
\includegraphics[width=\textwidth]{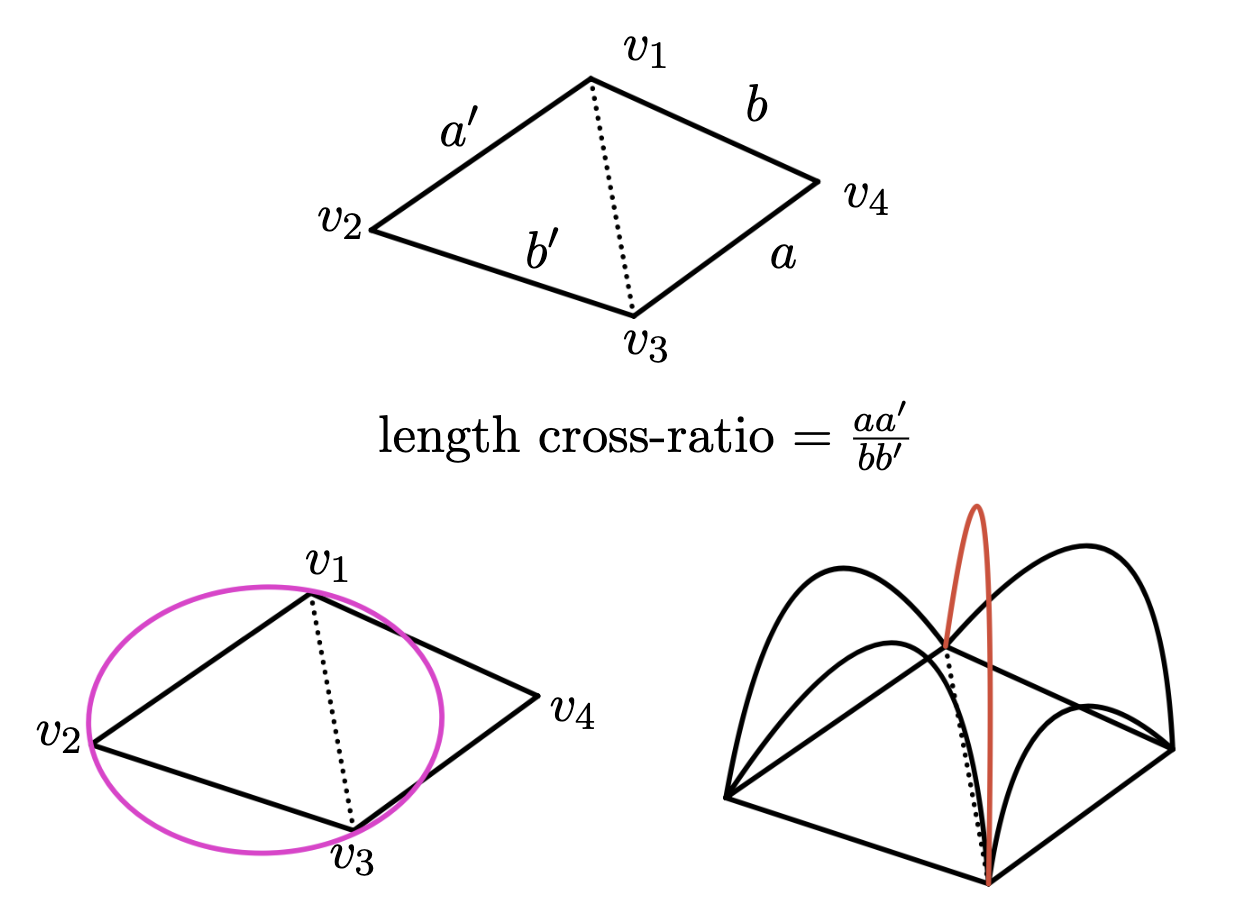}
\caption{Length cross-ratio, the empty-disk condition, and convex hulls.}
\label{ds19}

  \end{minipage}  
  \hspace{0.2cm}
  \begin{minipage}[t]{0.45\textwidth}
    \centering
    \includegraphics[width=0.95\textwidth]{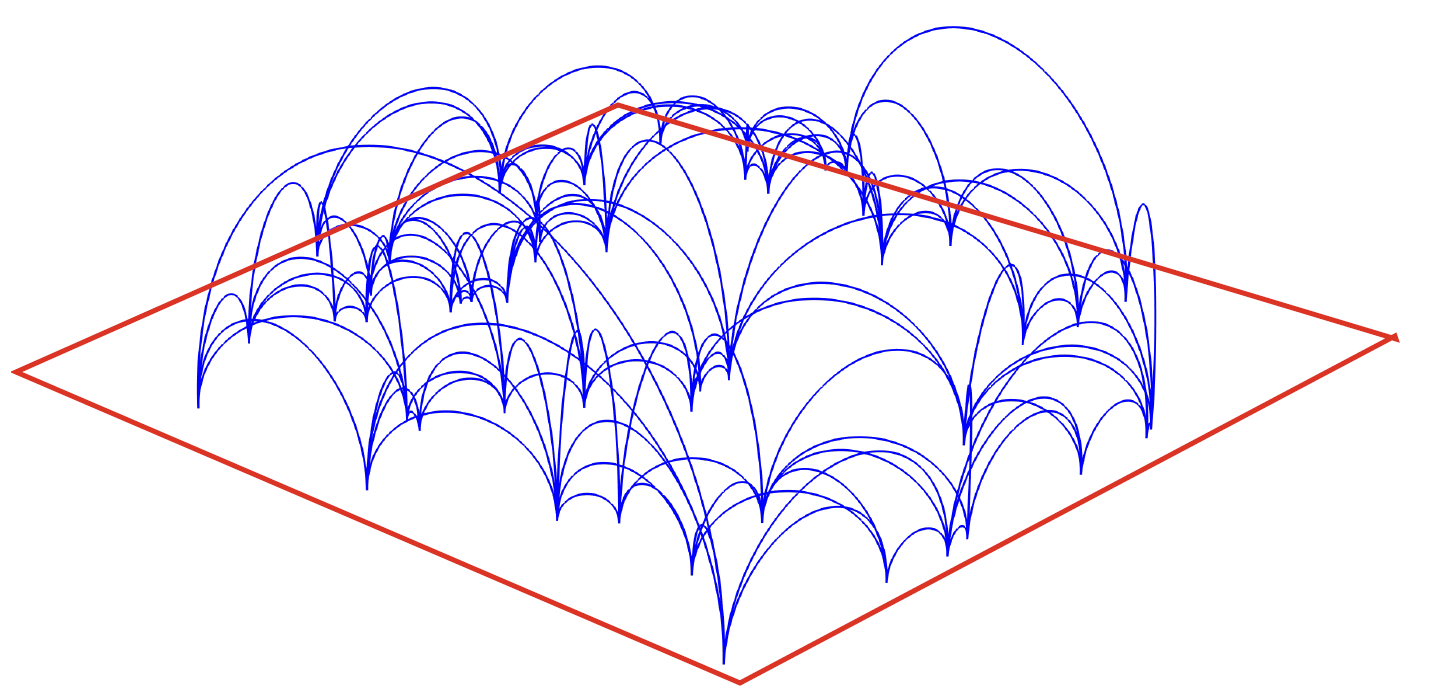}
     \caption{A Delaunay triangulation in $\C$ and the convex hull of its vertex set.}
    \label{dll7}
  \end{minipage}

\end{figure}

The second example is a discretization of conformal maps between complex planes, also known as Liouville-type theorems.
Recall that a geometric triangulation $\T$ of the plane is \it Delaunay \rm if the interior of the circumdisk of each triangle contains no vertices of $\T$. Delaunay triangulations play a central role in computational geometry and are closely related to the hyperbolic convex hulls. Indeed, if $V$ is the set of all vertices of a Delaunay triangulation $\T$ of $\C$,  codimension-1 faces in the hyperbolic convex hull $C(V \cup \{\infty\})$ in $\H^3_U$ are exactly the intersections of $C(V \cup \{\infty\})$ with the convex hulls of circumcircles of triangles in $\T$.  Furthermore, if $e=v_1v_3$ is an edge in $\T$ adjacent to two triangles $\triangle v_1v_2v_3$ and $\triangle v_1v_3v_4$, then the length-cross-ratio $\frac{|v_1-v_2||v_3-v_4|}{|v_2-v_3||v_4-v_1|}$ of $\T$ at $e$ is equal to the shear-coordinate of $\partial C(V \cup \{\infty\})$ at the corresponding hyperbolic geodesic ending at $v_1$ and $v_3$. It shows that the geometry of the hyperbolic surface $\partial C(V \cup \{\infty\})$ is determined by the length-cross-ratios at the edges of $\T$. See Figure \ref{dll7}.  As a consequence of Theorem \ref{main2}, we obtain the following discrete Liouville-type theorem

\begin{corollary} \label{cor16} Suppose $\T$ and $\T'$ are two Delaunay triangulations of the complex plane such that there exists a graph isomorphism between their 1-skeletons preserving the length-cross-ratios. Then $\T$ and $\T'$ differ by a complex linear isomorphism.  \end{corollary}

This result generalizes the previous work of \cite{WGS, DGM, DW24, Luo-sun-wu} where extra conditions such as angle bounds or special combinatorics are imposed for $\T$ and $\T'$.

Corollary \ref{cor16} should be compared with Schramm's rigidity theorem of infinite circle packings which is the circle packing version of Liouville's theorem.  It plays the key role in establishing the circle packing uniformization theorem (see \cite{stephenson, He-schramm, He1999}).  Recall that an infinite circle packing on the plane is an infinite collection of round disks with disjoint interiors such that each compact set intersects only a finite number of disks. The \it nerve \rm of the circle packing is the graph whose vertices correspond to disks and whose edges correspond to pairs of tangent disks. See Figure \ref{dsl6} and \ref{ds17}. Schramm's rigidity theorem states,

\begin{theorem}[Schramm, \cite{schramm}] \label{scharmm} If $P$ and $P'$ are two infinite circle packings on the plane whose nerves are isomorphic to the 1-skeleton of a triangulation of the plane, then $P$ and $P'$ differ by a complex linear transformation.   
\end{theorem}

\begin{figure}[htbp]
  \centering
  \begin{minipage}[t]{0.495\textwidth}
    \centering
\includegraphics[width=0.8\textwidth]{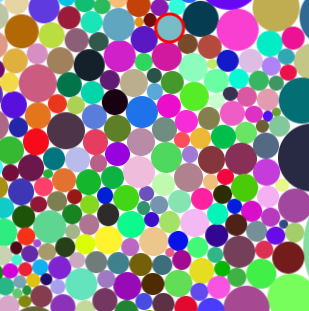}
    \caption{A circle packing on $\C$.} 
    \label{dsl6}
  \end{minipage} 
 \begin{minipage}[t]{0.495
 \textwidth}
    \centering
\includegraphics[width=0.85\textwidth]{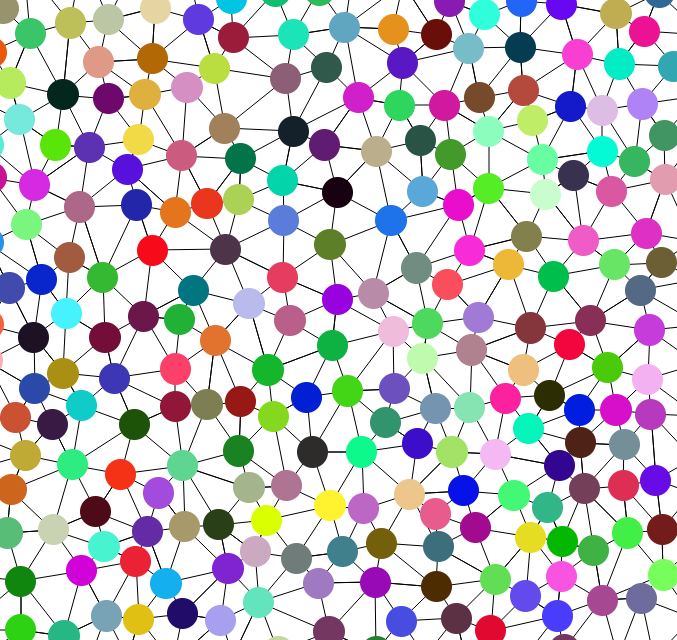}
     \caption{The nerve of a circle packing.} 
    \label{ds17}
  \end{minipage}
\end{figure}

In the simplest case, Schramm's theorem implies the Rodin--Sullivan theorem~\cite{rodin}, first conjectured by W. Thurston~\cite{Thurston1985},
that any hexagonal circle packing in the plane is regular. Corollary \ref{cor16} shows that any hexagonal Delaunay triangulation (i.e., the degree of each vertex is six) of the plane such that the length-cross-ratio of each edge is one is regular.

One of the main problems in the discrete conformal geometry is the discrete uniformization of all polyhedral surfaces, whether they are compact or not. As in the classical case, using the universal cover, the discrete uniformization problem in the vertex scaling setting can be reduced to the simply connected case.  The closed simply connected case corresponds to Corollary \ref{cor3} for finite sets $X$ and $Y$, first solved by I. Rivin~\cite{Rivin}.
For the remaining non-compact case, in the same manner as the classical approach to the uniformization theorem, one needs to show that any non-compact simply connected polyhedral surface is discrete conformal to a unique standard model which is either $(\C, V)$ or $(\D, V)$ where $V$ is a discrete subset of $\C$ or $\D$. Furthermore, in the case of the unit disk $\D$, one requires that $\partial \D$ is the set of all limit points of $V$.

The existence part of the discrete uniformization problem of non-compact simply connected polyhedral surfaces was established in \cite{LW24}.  The uniqueness part translates into the rigidity problem for non-compact convex sets in $\H^3$. In the classical case, the uniqueness follows from the Schwarz lemma and Liouville's theorem.  Analogously,  their counterparts in discrete conformal geometry correspond to the following two problems.  Namely, the discrete Liouville's theorem states that if $V_1$ and $V_2$ are two discrete subsets in the complex plane, considered to lie on the boundary at infinity of $\H^3_U$,  such that $\partial C(V_1\cup \{\infty\})$ is isometric to $\partial C(V_2\cup \{\infty\})$, then $V_1$ and $V_2$ differ by a M\"obius transformation. The discrete Schwarz lemma states that if $V_1$ and $V_2$ are two discrete subsets of $\D$ whose limit points are $\partial \D$ and $\partial C(V_1 \cup \D^c)$ is isometric to $\partial C(V_2 \cup \D^c)$, then $V_1$ and $V_2$ differ by a M\"obius transformation. Corollary  \ref{cor3} implies that the discrete Liouville theorem holds.  As a consequence, combining the work~\cite{LW24}, we obtain the following result concerning polyhedral surfaces.  Recall that every polyhedral surface has a natural complex structure.  A polyhedral surface is of \it parabolic type \rm if its universal cover is biholomorphic to $\C$.  

\begin{corollary} \label{co14} Suppose $(P, V, d)$ is a connected polyhedral surface with vertex set $V$ and cone metric $d$ whose underlying complex structure is of parabolic type. Then $(P, V,  d)$ is discrete conformal to a unique surface $(P, V, d')$ where  $d'$ is a complete flat metric on $P$.    
\end{corollary}

As the argument requires an extended discussion of polyhedral surfaces and discrete conformality, we refer the reader to \cite{LL26} for a proof that Corollary \ref{co14} follows from Theorem \ref{main2}, and for a comprehensive account of discrete conformal geometry.

\bigskip
\noindent{\bf Main ideas of the proof.}
The proof of Theorem \ref{main2} splits into three parts, based on the dimensions of the $X_i$. 

When both $X_1$ and $X_2$ are 3-dimensional, 
we use the Pogorelov map which sends a pair of locally convex isometric surfaces in $\H^3$ to a pair of locally convex isometric surfaces in $\R^3$. In our case, since the conformal ends of the convex surfaces in $\H^3$ are small due to the zero-length condition, we are able to compactify the two locally convex surfaces in $\R^3$ to obtain two compact convex surfaces without boundary and extend the local isometry between them to an isometry. Once this major step is achieved, Pogorelov's Theorem  \ref{pogo1} implies Theorem \ref{main2}.  The most technical part of the proof depends on the basic properties of the Pogorelov map and the Tabor--Tabor Theorem \ref{thm:Tabor} which states that if $A$ is a closed subset of an open set $U \subset \R^n$ with vanishing $(n-1)$-dimensional Hausdorff measure, then any locally convex function $f: U -A \to \R$ extends to a locally convex function on $U$.

When both $X_1$ and $X_2$ are 2-dimensional, we first demonstrate that the isometry between the Alexandrov doubles $D(X_i)$ must map each copy of $X_1$ to a corresponding copy of $X_2$. We establish this by analyzing the image of the boundary of $X_1$ in $D(X_2)$. Once we know $X_1$ is isometric to  $X_2$, the result follows.

Finally, we show that the case in which exactly one of the $X_i$ is 3-dimensional cannot happen. Suppose otherwise that $X_1$ is 3-dimensional and $X_2$ is two-dimensional. We prove that there exists a supporting plane of $X_1$ which intersects $X_1$ at a single point $p$. By generalizing a theorem of Busemann~\cite[Theorem 11.9]{Busemann}, we show that $p$ possesses a compact \emph{convex-cap} neighborhood in $\partial X_1$.
Doubling this convex cap yields a closed convex surface isometric to the Alexandrov double of a 2-dimensional convex set in $\H^2$. This, however, is precluded by Pogorelov's rigidity theorem (Theorem \ref{pogo1}).

\bigskip
\noindent{\bf An open problem.}\label{sec:open}
A natural question arising from this investigation corresponds to the Painlev\'e  problem in the convex surface setting. The classical Painlev\'e problem~\cite{Younsi} asks which compact sets in the complex plane are removable for bounded holomorphic functions. Using the work of Ahlfors on analytic capacity and the work of Vitushkin, Xavier Tolsa completely solved the Painlev\'e problem at the turn of the century (see \cite{Tolsa}).  The natural corresponding problem is the following.

\begin{problem} \label{P1} Find the necessary and sufficient condition on a compact set $X \subset \mathbb{S}^2$ such that if $Y\subset \mathbb{S}^2$ is a compact set with $\partial C(Y)$ isometric to $\partial C(X)$ in $\H^3$, then $X$ and $Y$ differ by a M\"obius transformation.
 
\end{problem}

Similar to the Painlev\'e problem, Thurston’s work implies that a necessary condition for $X$ is  that it be totally disconnected, 
i.e., each connected component must be a single point. A well-known counterexample to the rigidity part of Koebe's circle domain conjecture shows that there exist two Cantor sets $X$ and $Y$ in $\C$ whose complements are conformal and $X$ and $Y$ are not related by any M\"obius transformation. In this case, both $X$ and $Y$ have positive 2-dimensional Lebesgue measure and are constructed using the measurable Riemann mapping theorem.
We believe there are two Cantor sets $X$ and $Y$ in $\mathbb{S}^2$ such that $\partial C(X)$ is isometric to $\partial C(Y)$ and $X$ and $Y$ are not related by M\"obius transformations.  

The proof of Theorem \ref{main2} shows that Problem \ref{P1} is related to the removable sets for locally convex functions defined on open sets in $\R^n$. See the works of Tabor--Tabor~\cite{TT10} and Pokorn\'y--Rmoutil~\cite{PR-removable} for more details.

\bigskip
\noindent{\bf Organization.} 
The paper is organized as follows.  In Section \ref{sec:preliminary}, we recall the basics about path metrics, convex sets,  convex surfaces, the Minkowski space and hyperbolic spaces.  
In Section \ref{sec:Pogorelov}, we recall the basics about the Pogorelov maps.  Section \ref{sec:extension} proves several lemmas on path metric surfaces and recalls the Tabor--Tabor theorem.  Theorem \ref{main2} is proved in Section \ref{sec:main-proof}.  

\bigskip
\noindent{\bf Acknowledgments.} We thank Xiaochun Rong and Guofang Wei for discussions.
Figures \ref{dsl6} and  \ref{ds17} were produced using Circle Packing V.10 developed by B. Beeker and B. Loustau. The work is supported in part by NSF DMS 2220271, NSF DMS 2501286,  and the Simons Foundation grant SFI-MPS-SFM-00011051.

\smallskip

\section{Preliminaries}\label{sec:preliminary}
We will quickly review some basic concepts and results to be used in the paper. These include path metrics, local isometries, locally convex functions, locally convex sets, locally convex surfaces, and complete convex surfaces in the Euclidean and hyperbolic 3-spaces. 

The length $l(\gamma)$ of a path $\gamma: [a, b] \to X$ in a metric space $(X, d)$ is defined to be $$\sup \{ \sum_{ i=0}^n d(\gamma(t_i),\gamma(t_{i+1})):  a=t_0<t_1<...< t_{n+1}=b\}.$$  Here, the supremum is taken over all partitions of the interval $[a,b]$. A path of finite length is called \emph{rectifiable}. A \emph{path metric space}  $(X,d)$ is a metric space such that for all $x, y \in X$, 
$$d(x,y) = \inf\{ l(\gamma): \gamma \text{ is a path from $x$ to $y$}\}.$$  A distance-minimizing path $\gamma$ in a path metric space is a rectifiable path whose length is equal to the distance between its endpoints. A locally distance-minimizing path is called a \emph{geodesic}.  See \cite{BBI} for more details.

If $p \in X$ is a point in a metric space $(X, d)$, the open ball $\{x \in X : d(x, p)<r\}$ of radius $r$ centered at $p$ is denoted by $B_p(r)$.  

The following simple lemma will be used later. Let $(X, d_X)$ and $(Y, d_Y)$ be two path metric spaces.  Recall a \it local isometry \rm  $g: X\to Y$ is a  map with the property that for any $x\in X$, there exists a neighborhood $U$ of $x$ such that $g|_U: U \to g(U)$ is an isometry, i.e. $d_X(u_1, u_2) = d_Y(g(u_1), g(u_2))$. 

\begin{lemma}
\label{local+homeo=global}
    Let $g: X \to Y$ be a local isometry between path metric spaces $(X, d_X)$ and $(Y, d_Y)$. If $g$ is a homeomorphism, then $g$ is an isometry. 
\end{lemma}
\begin{proof} Since $g$ is a homeomorphism, $g$ sends an open set $U$ to an open set $g(U)$. In particular, it implies that $g^{-1}$ is a local isometry.
 To show that $g$ is an isometry, it suffices to prove that for any rectifiable path $\gamma$ in $X$, $\gamma$ and $g(\gamma)$ have the same length.  This follows from the additivity of the lengths of paths. Namely, if $\gamma: [a, b] \to X$ is a path and $t_0=a<t_1<...<t_n=b$ is a partition of the interval, then the length $l(\gamma)$ is equal to the sum of the lengths of paths $\gamma|_{[t_i, t_{i+1}]}$ for $i=0, 1, ..., n-1$. Using Lebesgue's number lemma for the compact set $\gamma([a,b])$ covered by those open sets $U$ on which the restrictions $g|_U$ are isometries, we find a partition $a=t_0< t_1 <...<t_n=b$  such that for each $i$, $\gamma([t_i, t_{i+1}])$ is contained in an open set $U$ on which $g|_U$ is an isometry.  Then, by the definition of a local isometry, we see that $g\circ \gamma|_{[t_i, t_{i+1}]}$ is rectifiable of length $l(\gamma|_{[t_i, t_{i+1}]})$. Therefore, by additivity, we obtain
\begin{equation*}
    l(\gamma) =\sum_{i=0}^{n-1} l(\gamma|_{[t_i, t_{i+1}]}) = \sum_{i=0}^{n-1} l(g\circ \gamma|_{[t_i, t_{i+1}]})  =l(g(\gamma)).\qedhere
\end{equation*}  
\end{proof}

We use $\H^n$, $\R^n$, and $\mathbb{S}^n$ to denote the $n$-dimensional hyperbolic space, the Euclidean space, and the n-dimensional sphere, respectively.  Convex sets in $\H^n$ and $\R^n$ are defined in the standard way.  

By a geodesic segment in $\H^n$ we mean an interval in a geodesic in $\H^n$. The interval could be open, closed, or half-open. If $p, q$ are two distinct points in $\H^n \cup \partial \H^n$, we use $[p,q]$ to denote the geodesic in $\H^n$ joining $p$ to $q$.

If $\Sigma$ is a surface with boundary, we use $\mathring{\Sigma}$ to denote its interior. 
A surface is \it closed \rm if it is compact and has no boundary.
If $X$ is a closed convex three-dimensional set in $\H^3$ or $\R^3$, its boundary $\partial X$ is called a complete convex surface.
We call $X$ the \it convex side \rm of the surface $\partial X$.  The surface $\partial X$ is a complete path metric space under its intrinsic metric. Here, completeness means every Cauchy sequence has a limit point. A theorem of Alexandrov~\cite{alex} shows that for any two points $x,y$ on a complete convex surface, there exists a distance-minimizing path joining the two points.  

A subset $X$ in $\H^3$ or $\R^3$ is  \it locally convex \rm if for each $p \in X$, there exists a ball $B_r(p)$ in  $\H^3$ or $\R^3$  such that $B_r(p) \cap X$ is convex.  A surface $S \subset  \H^3$ or $\R^3$ is \it locally convex \rm if for each point $p \in S$, there exists a 3-dimensional convex set $X_p$ such that $\partial X_p \cap S$ contains a neighborhood of $p$ in $S$.  In particular, a surface contained in the boundary of a 3-dimensional convex set in $\H^3$ or $\R^3$ is locally convex. 

\smallskip

\section{Basic Properties of the Pogorelov map}\label{sec:Pogorelov}

In this section, we review the key tool of this paper, the Pogorelov map, which was originally defined in the spherical geometry $\mathbb S^n$ by Pogorelov~\cite{Pogorelov}. Pogorelov mentioned that the same map is defined on the hyperbolic space $\H^n$, but he did not carry out the detailed computation in order ``To keep the exposition as simple and clear as possible" (line 3, page 271, in \cite{Pogorelov}). Since the Pogorelov map on $\H^n$ is so essential to our proof and the relevant results are hard to find in the literature, we will collect all of its basic properties that will be used in this section for completeness.  Some of the proofs of these properties may be new.
For references, see \cite{Pogorelov, Spivak, Schlenker00, Schlenker06, Izmestiev, Virto, MS12, Izmestiev2}.

\subsection{The Pogorelov map on Minkowski spaces}

 Let $(\mathbb{R}^{1, n}, \lan \cdot, \cdot\ran  )$ be the Minkowski space with the Minkowski product for ${x} = (x_0, x_1, \dots, x_n) $ and $y = (y_0, y_1, \dots, y_n)$ given by
\begin{equation}
    \lan  x,  y\ran   = -x_0y_0 + \sum_{i=  1}^nx_iy_i = -x_0y_0 + \PP(x) \cdot \PP(y),
\end{equation}
where $\PP(x_0, x_1, ..., x_n)=(x_1, ..., x_n)$ and $u \cdot v$ is the Euclidean dot product of two vectors $u, v \in \R^n.$  
A linear transformation of $\mathbb R^{1,n}$ preserving the Minkowski product is called a Lorentz transformation.

Let $e = (1, 0,\cdots, 0)\in \R^{1,n}$.  Then, $$(0,\PP(x)) =x +\lan x, e \ran e.$$
Since $\lan x, e \ran =-x_0$, 
we have, 
\begin{equation}\label{basicpdt}
\begin{aligned} 
\PP(x) \cdot \PP(y) = \lan x, y\ran   + x_0y_0 = \lan x, y \ran+ \lan x, e \ran  \lan y, e\ran. 
    \end{aligned}
\end{equation}

The hyperboloid model $\mathcal{H}^n$ of the $n$-dimensional hyperbolic space is  given by 
\begin{equation}
    \mathcal{H}^n = \{x\in\mathbb{R}^{1,n} : \lan x, x\ran   = -1, x_0>0  \} = \{x\in\mathbb{R}^{1,n} : x_0^2 = 1 + \sum_{i=1}^n x_i^2, x_0>0\}.
\end{equation}

\begin{definition}[\cite{Pogorelov}]
    The Pogorelov map $\Phi\from \mathbb{R}^{1, n} \times \mathbb{R}^{1,n}  -\{(x, y) \in \mathbb{R}^{1,n}\times \mathbb{R}^{1,n}: \langle x+y, e\rangle = 0\} \to \mathbb{R}^n\times \mathbb{R}^n$ is defined by
    \begin{equation}
        \Phi(x, y) = \left(\frac{2\PP(x)}{-\lan x+y,e\ran}, \frac{2\PP(y)}{-\lan x+y,e\ran}\right) .
    \end{equation}
In particular, the Pogorelov map $\Phi$ is defined on $\mathcal H ^n \times \mathcal {H} ^n$. Let $P_i$ be the components of the Pogorelov map $\Phi$, i.e., 
\begin{equation}\label{P_i}
    P_1(x, y) = \frac{2\PP(x)}{-\lan x+y,e\ran  }, \quad P_2(x, y) = \frac{2\PP(y)}{-\lan x+y,e\ran  }.
\end{equation}
\end{definition}

The following lemma illustrates the length-preserving property of the Pogorelov map $\Phi$.

\begin{lemma}\label{lem:equal-length}  Let $x, y, u, v\in \mathbb{R}^{1,n}$ such that $\lan x, x\ran   = \lan y, y\ran  $ and $\lan u, u\ran   = \lan v, v\ran  $. Suppose $\hat x = P_1(x, y)$, $\hat y = P_2(x, y)$, $\hat u = P_1(u, v)$, and $\hat v = P_2(u, v)$ are defined. Then 
    $\|\hat x - \hat u\| = \|\hat y - \hat v\|$ if and only if $\lan x,u\ran   =\lan y,v\ran  $, where $\|\cdot\|$ is the Euclidean norm.
\end{lemma}
\begin{proof}
    Note that by \eqref{basicpdt} 
    \begin{equation}
        \hat x \cdot \hat u =  \frac{4\PP(x)\cdot\PP(u)}{\lan x+y, e\ran \lan u+v, e\ran } =  \frac{4\lan x+\lan x, e\ran e, u+\lan u, e\ran e \ran}{\lan x+y, e\ran \lan u+v, e\ran } =\frac{4(\lan x,u \ran+\lan x, e\ran \lan u, e\ran)}{\lan x+y, e\ran \lan u+v, e\ran}.
    \end{equation}
   Similarly,  we have 
 \begin{equation}
            \hat x \cdot \hat x  = \frac{4}{\lan x+y, e\ran  ^2}(\lan x, x\ran   + \lan x, e\ran  ^2), \quad \hat u\cdot \hat u = \frac{4}{\lan u+v, e\ran  ^2}(\lan u, u\ran   + \lan u, e\ran  ^2),
    \end{equation}
    and the similar identities for $\hat y\cdot \hat y$, $\hat v\cdot \hat v$, and $\hat y\cdot \hat v$.  Then
    \begin{equation}
        \begin{aligned}
            \|\hat x - \hat u\|^2 - \|\hat y- \hat v\|^2 = &\
            \frac{4}{\lan x+y, e\ran  ^2}(\lan x, e\ran  ^2 - \lan y, e\ran  ^2) +  \frac{4}{\lan u+v, e\ran  ^2}(\lan u, e\ran  ^2 - \lan v, e\ran  ^2)\\
            & -\frac{8}{\lan x+y,e\ran  \lan u+v,e\ran  } (\lan x, e\ran  \lan u, e\ran   - \lan y, e\ran  \lan v, e\ran  +\lan x,u\ran  -\lan y,v\ran  )\\
            = &\ 4 \left(\frac{x_0-y_0}{x_0+y_0} + \frac{u_0-v_0}{u_0+v_0} - \frac{2(x_0u_0-y_0v_0)}{(x_0+y_0)(u_0+v_0)}\right)\\
            & - \frac{8}{\lan x+y,e\ran  \lan u+v,e\ran  } \left(\lan x,u\ran  -\lan y,v\ran\right)\\
            = & - \frac{8}{\lan x+y,e\ran  \lan u+v,e\ran  } \left(\lan x,u\ran  -\lan y,v\ran\right).
        \end{aligned}
    \end{equation}
    Thus, the statement is proved.
\end{proof}

Let $d_\HH$ be the hyperbolic distance in $\mathcal H^n $. Then for any $x,u\in \mathcal H^n $, we have $$\lan x,u\ran   =-\cosh\left(d_{\mathcal{H}}(x,u)\right).$$ The above computation has the following consequence.

\begin{corollary}
    Suppose $x, y, u, v\in \mathcal H^n $. Let $\hat x = P_1(x, y)$, $\hat y = P_2(x, y)$, $\hat u = P_1(u, v)$, and $\hat v = P_2(u, v)$. Then $||\hat x - \hat u|| > ||\hat y - \hat v||$ if and only if $d_\HH(x,u) > d_\HH(y,v)$.
\end{corollary}

Lemma \ref{lem:equal-length} also implies that $\Phi$ maps a pair of geodesic segments of the same length in $\mathcal H^n$ to a pair of line segments of the same length in $\R^n$. 

\begin{lemma}
    Suppose $\alpha(t)$ and $\beta(t)$ are two unit-speed geodesics of the same length in $\mathcal{H}^n$ with $t\in[0,l]$. Let $\gamma(t)=P_1(\alpha(t), \beta(t))$ and $\delta(t)=P_2(\alpha(t), \beta(t))$. 
    Then $\gamma$ and $\delta$ are two line segments of the same length in $\R^n$.
\end{lemma}
We remark that the parametrized curves $\gamma(t)$ and $\delta(t)$ may not be of constant speed.

\begin{proof}
    Let $x = \alpha(0)$, $u = \alpha(l)$, $y = \beta(0)$, $v = \beta(l)$, $z = \alpha(l/2)$ and $w = \beta(l/2)$. We denote the corresponding points on $\gamma$ and $\delta$ for $t=0,l,l/2$ by $\hat x, \hat y, \hat u, \hat v, \hat z$ and $\hat w$. More precisely, $\hat x =P_1(x,y)$, $\hat y=P_2(x,y)$, and so on. Since $d(x,u)=d(y,v)=l$, and $d(x,z)=d(y,w)=l/2$, we have $\lan x, u\ran  =\lan y, v\ran  $, $\lan x, z\ran  =\lan y, w\ran  $.  By Lemma \ref{lem:equal-length}, we have $||\gamma(0)-\gamma(l)||=||\delta(0)-\delta(l)||$.  Therefore, it remains to show that $\gamma$ and $\delta$ are parametrized line segments.  To prove this, it suffices to prove that $\hat z$ is on the line segment connecting $\hat x$ and $\hat u$.

    By the definition of geodesics in $\mathcal H^n $ and the condition that $z$ is the midpoint of $x,u$, there exists $\lambda >0$ such that $z = \lambda(x+u)$. Similarly, there exists $\mu>0$ such that  $w = \mu(y+v)$. Now 
    \begin{equation}
        -1 = \lan z, z\ran   = \lambda^2\lan x+u, x+u\ran = \lambda^2 (-2+2\lan x,u\ran).
    \end{equation}
    Similarly, we have $-1=\mu^2 (-2+2\lan y,v\ran).$ Then by $\lan x, u\ran  =\lan y, v\ran$, we have $\lambda=\mu$.
    Therefore,
    \begin{equation}
    \begin{aligned}
        \hat z & = \frac{2\PP(z)}{-\lan z+w,e\ran  } = \frac{2\PP(\lambda(x+u))}{-\lan \lambda(x+u+y+v),e\ran  }=\frac{2\PP(x+u)}{-\lan x+u+y+v,e \ran}\\
        & = \frac{\lan x+y, e\ran  }{\lan x+y+v+u,e\ran  }\cdot\frac{2\PP(x)}{-\lan x+y, e\ran  } + \frac{\lan v+u, e\ran  }{\lan x+y+v+u,e\ran  }\cdot\frac{2\PP(u)}{-\lan u+v, e\ran  }\\
        & = \frac{\lan x+y, e\ran  }{\lan x+y+v+u,e\ran  } \hat x + \frac{\lan v+u, e\ran  }{\lan x+y+v+u,e\ran  } \hat y.
    \end{aligned}
    \end{equation}
    Let $t={\lan x+y, e\ran  }/{\lan x+y+v+u,e\ran  }$, then $t\in[0,1]$ and $\hat z = t\hat x +(1-t) \hat y$.
\end{proof}

A version of the above lemma for rectifiable curves takes the following form. 

\begin{lemma}\label{lem:path}
    Suppose $\alpha(t)$ and $\beta(t)$ are two  unit-speed rectifiable curves of the same length in $\mathcal{H}^n$ with $t\in[0,l]$. Let $\gamma(t)=P_1(\alpha(t), \beta(t))$ and $\delta(t)=P_2(\alpha(t), \beta(t))$. Then $\gamma(t)$ and $\delta(t)$ are two curves in $\R^n$ with the same speed, i.e., $\|\gamma'(t)\|=\|\delta'(t)\|$. In particular, $\gamma$ and $\delta$ have the same length.
\end{lemma}
\begin{proof} Let us begin with the case that both $\alpha$ and $\beta$ are smooth. 
    Suppose that
    $$\alpha(t)=(\alpha_0(t),\alpha_1(t),\dots,\alpha_n(t)),\quad \beta(t)=(\beta_0(t),\beta_1(t),\dots,\beta_n(t)).$$ By definition, we have
    \begin{equation}
        \gamma(t)=\frac{2(\alpha_1(t),\dots,\alpha_n(t))}{\alpha_0(t)+\beta_0(t)},   \quad \delta(t)=\frac{2(\beta_1(t),\dots,\beta_n(t))}{\alpha_0(t)+\beta_0(t)}.
    \end{equation}
    Using the quotient rule, we obtain,
    \begin{equation}\label{eq:gamma}
        \|\gamma'(t)\|^2=\gamma'(t)\cdot\gamma'(t)=\frac{4}{(\alpha_0+\beta_0)^4}\sum\limits_{i=1}^n \left(\alpha'_i(\alpha_0+\beta_0)-\alpha_i(\alpha'_0+\beta'_0)\right)^2.
    \end{equation}
    Similarly, we have
    \begin{equation}\label{eq:delta}
        \|\delta'(t)\|^2=\frac{4}{(\alpha_0+\beta_0)^4}\sum\limits_{i=1}^n (\beta'_i(\alpha_0+\beta_0)-\beta_i\left(\alpha'_0+\beta'_0)\right)^2.
    \end{equation}
    Since $\alpha(t),\beta(t)\in\mathcal H^n $, we have $\lan\alpha(t),\alpha(t)\ran=\lan\beta(t),\beta(t)\ran=-1$. Hence,
    \begin{equation}
       \lan\alpha'(t),\alpha(t)\ran=\lan\beta'(t),\beta(t)\ran=0.
    \end{equation}
    On the other hand, $\alpha(t)$ and $\beta(t)$ are arc-length parametrized, so we have $$\lan \alpha'(t),\alpha'(t)\ran=\lan \beta'(t),\beta'(t)\ran=1.$$
    Therefore,
    \begin{equation}  \label{sums1}
    \begin{aligned}
        &\sum\limits_{i=1}^n \alpha_i^2=\alpha_0^2-1, \quad \sum\limits_{i=1}^n \beta_i^2=\beta_0^2-1, \\
        &\sum\limits_{i=1}^n \alpha_i\alpha'_i=\alpha_0\alpha'_0, \quad \sum\limits_{i=1}^n\beta_i\beta'_i=\beta_0\beta'_0,\\
        &\sum\limits_{i=1}^n (\alpha'_i)^2=(\alpha'_0)^2+1, \quad \sum\limits_{i=1}^n (\beta'_i)^2=(\beta'_0)^2+1. 
    \end{aligned}
    \end{equation}
    Using \eqref{sums1}, we can simplify the numerator 
    in (\ref{eq:gamma}) as, 
    \begin{align}
        &\sum\limits_{i=1}^n (\alpha'_i(\alpha_0+\beta_0)-\alpha_i(\alpha'_0+\beta'_0))^2\\
        =&\ (\alpha_0+\beta_0)^2\sum\limits_{i=1}^n (\alpha'_i)^2 +(\alpha'_0+\beta'_0)^2 \sum\limits_{i=1}^n(\alpha_i)^2 - 2(\alpha'_0+\beta'_0)(\alpha_0+\beta_0) \sum\limits_{i=1}^n \alpha_i\alpha'_i\\
        =&\ (\alpha_0+\beta_0)^2\bigl((\alpha'_0)^2+1\bigr) + (\alpha'_0+\beta'_0)^2(\alpha_0^2-1) - 2(\alpha'_0+\beta'_0)(\alpha_0+\beta_0) \alpha_0\alpha'_0\\
        =&\ (\alpha_0')^2\beta_0^2+\alpha_0^2(\beta_0')^2-2\alpha_0\alpha_0'\beta_0\beta_0' +(\alpha_0+\beta_0)^2
        -(\alpha_0'+\beta_0')^2\\
        = &\ (\alpha_0\beta_0'-\beta_0\alpha_0')^2+ (\alpha_0+\beta_0)^2-(\alpha_0'+\beta_0')^2,
    \end{align}
which does not depend on $\alpha_i$ and $\beta_i$ for $i = 1, \cdots, n$.
Moreover, this expression is symmetric with respect to $\alpha_0$ and $\beta_0$. Hence, we have
    $$\sum\limits_{i=1}^n \left(\alpha'_i(\alpha_0+\beta_0)-\alpha_i(\alpha'_0+\beta'_0)\right)^2 = \sum\limits_{i=1}^n \left(\beta'_i(\alpha_0+\beta_0)-\beta_i(\alpha'_0+\beta'_0)\right)^2,$$
    and then $\|\gamma'(t)\|^2=\|\delta'(t)\|^2$ from \eqref{eq:gamma} and \eqref{eq:delta}.

Now, if $\alpha(t)$ and $\beta(t)$, $t \in [0, L]$,  are rectifiable, we approximate them by unit-speed smooth curves $\alpha_n(t)$ and $\beta_n(t)$ with $t \in [0, L_n]$ such that $\lim_n L_n =L$. The lemma follows from the smooth case and taking the limit.   
\end{proof}

\subsection{The Pogorelov map on the hyperboloid $\mathcal{H}^n$ }

In this subsection, we illustrate some fundamental properties of the Pogorelov map $\Phi$ when the domain of $\Phi$ is restricted to $\mathcal H^n \times \mathcal H^n $.  Note that if $x \in \mathcal H^n$, then $-\lan x,e\ran=\sqrt{1+\|\PP(x)\|^2} >\|\PP(x)\|$. Hence, for any $x,y\in\mathcal H^n $, 
    \begin{equation}
    \label{boundness}
        \|P_1(x,y)\|+\|P_2(x,y)\|=\frac{2\|\PP(x)\|+2\|\PP(y)\|}{-\lan x,e\ran-\lan y,e\ran}<2. 
    \end{equation}
 It implies that the image of $\mathcal H^n \times\mathcal H^n $ under $\Phi(x,y) =(P_1(x,y), P_2(x,y))$ is in $\Omega = \{(\hat x, \hat y)\in \mathbb{R}^n\times \mathbb{R}^n: \|\hat x\| + \|\hat y\| < 2\}. $  It turns out the map $\Phi|:  \mathcal H^n \times \mathcal H^n \to \Omega$ is a diffeomorphism.  To show this, we define the function $f(a,b)$ for $(a,b) \in \Omega$ by \begin{equation}
    \begin{aligned}
        f(a, b) & = \|a\|^4-2\|a\|^2\|b\|^2+\|b\|^4-8\|a\|^2-8\|b\|^2+16\\
        & = (||a||^2 - ||b||^2)^2 - 8(||a||^2 + ||b||^2) + 16\\
        & = (2+\|a\|+\|b\|)(2-\|a\|+\|b\|)(2+\|a\|-\|b\|)(2-\|a\|-\|b\|).
    \end{aligned}
\end{equation}
Note that $f(a,b)>0$. 

\begin{proposition}[Proposition 6.3, \cite{Virto}]\label{prop:Pogorelov-inverse}
    The map $\Psi\from\Omega \to \mathcal{H}^n\times\mathcal{H}^n$ given by
    \begin{equation}
        \Psi(a, b) = ((\sqrt{\|\hat x\|^2 + 1 }, \hat x),(\sqrt{\|\hat y\|^2 + 1 }, \hat y) ),
    \end{equation}
    is the inverse of the Pogorelov map $\Phi|_{\mathcal H^n \times\mathcal H^n }$, where
    \begin{equation}
        \hat x = \frac{4a}{\sqrt{f(a,b)}}, \hat y = \frac{4b}{\sqrt{f(a,b)}}.
    \end{equation}
   In particular, the Pogorelov map $\Phi|_{\mathcal H^n \times \mathcal H^n}$ is a diffeomorphism from $\mathcal{H}^n\times \mathcal{H}^n$ to $\Omega$.
\end{proposition}
Since the manuscript~\cite{Virto} has not yet been published, we present Virto's proof for completeness. 
\begin{proof} 
    It is a straightforward computation to verify that $\Psi(\Phi(x,y))=(x,y)$ and $\Phi(\Psi(a,b))=(a,b)$, for any $x,y\in\mathcal H^n $ and $(a,b)\in\Omega$. Indeed,  let $a  = -{2\PP(x)}/{\lan x+y,e\ran}$ and $b =  -{2\PP(y)}/{\lan x+y,e\ran}$. Then 
    $$||a||^2 = \frac{4 \lan x+\lan x, e\ran e, x+\lan x, e\ran e \ran}{\lan x+y,e\ran^2} = \frac{4(\lan x, e\ran^2 - 1)}{\lan x+y,e\ran^2}, \quad ||b||^2 = \frac{4(\lan y, e\ran^2 - 1)}{\lan x+y,e\ran^2},$$
    and 
    $$f(a, b) = \frac{16}{\lan x+y,e\ran^4}(\lan x, e\ran^2 - \lan y, e\ran^2)^2 - \frac{32}{\lan x+y,e\ran^4}(\lan x, e\ran^2 + \lan y, e\ran^2-2) + 16$$
    $$ = \frac{16}{\lan x+y,e\ran^2}\Big((\lan x, e\ran - \lan y, e\ran)^2 - 2\lan x, e\ran^2 - 2\lan y, e\ran^2 +4 + (\lan x, e\ran + \lan y, e\ran)^2    \Big) =\frac{64}{\lan x+y,e\ran^2}.$$
    Then we have
    $$\hat x = \frac{4a}{\sqrt{f(a,b)}} = {\frac{8\PP(x)}{-\lan x+y,e\ran}}\Big/{\frac{8}{-\lan x+y,e\ran}} = \PP(x),$$ 
    $$\hat y = \frac{4b}{\sqrt{f(a,b)}} = {\frac{8\PP(y)}{-\lan x+y,e\ran}}\Big/{\frac{8}{-\lan x+y,e\ran}} = \PP(y).$$
    Therefore
    $$\sqrt{\|\hat x\|^2 + 1 } = \sqrt{\|\PP(x)\|^2 + 1} = x_0, \quad \sqrt{\|\hat y\|^2 + 1 } = \sqrt{\|\PP(y)\|^2 + 1} = y_0.$$
    This shows $\Psi(\Phi(x,y))=(x,y)$. 
    
    Conversely, by definition, $\Psi(a, b) = (x, y)$, where
    \begin{equation}\label{eq:converse}
        x = \left(\sqrt{\frac{16||a||^2}{f}+1}, \frac{4a}{\sqrt{f}}\right), \quad y = \left(\sqrt{\frac{16||b||^2}{f}+1}, \frac{4b}{\sqrt{f}}\right).
    \end{equation}
    Notice that 
    $$\sqrt{\frac{16||a||^2}{f}+1} = \sqrt{\frac{(||a||^2 - ||b||^2)^2 +8(||a||^2 - ||b||^2) + 16}{f}}  
    = \frac{||a||^2 - ||b||^2 + 4}{\sqrt{f}},$$
    and
    $$\sqrt{\frac{16||b||^2}{f}+1} = \sqrt{\frac{(||a||^2 - ||b||^2)^2 +8(||b||^2 - ||a||^2) + 16}{f}}  
    =\frac{||b||^2 - ||a||^2 + 4}{\sqrt{f}}.$$
    We have 
    $$-\lan x+y,e\ran =\frac{||b||^2 - ||a||^2 + 4}{\sqrt{f}} + \frac{||a||^2 - ||b||^2 + 4}{\sqrt{f}} = \frac{8}{\sqrt{f}}. $$
    Therefore,  
    \begin{equation*}
        \Phi(\Psi(a,b)) = \Phi(x, y) = \left({2\cdot\frac{4a}{\sqrt{f}}}\Big/{\frac{8}{\sqrt{f}}},2\cdot\frac{4b}{\sqrt{f}}\Big/{\frac{8}{\sqrt{f}}} \right) = (a, b).\qedhere
    \end{equation*}
\end{proof}

\subsection{Congruency-preserving property}

The following proposition comes from \cite{Virto}, and here we present a different proof. This property was originally established by Pogorelov \cite{Pogorelov}, where he showed that two isometric surfaces are congruent in the spherical geometry $\mathbb S^n$  if and only if their images are congruent in the Euclidean space. 

\begin{proposition}[Proposition 6.5, \cite{Virto}]\label{prop:AtoB}
    For any $A\in \Isom(\mathcal{H}^n)$, there exists a unique $B\in \Isom(\mathbb{R}^n)$ such that for any $x\in \mathcal{H}^n$, 
    $\Phi(x, Ax) = (y, By)$ and $\Phi(x, A^{-1}x) =(z, B^{-1}z)$.
\end{proposition}

\begin{proof}
    Let $(y, y') = \Phi(x, Ax)$, i.e., 
   $$y = \frac{2\PP(x)}{-\lan x+Ax, e\ran}, \quad y' = \frac{2\PP(Ax)}{-\lan x+Ax, e\ran}.$$
By identifying $\{0\} \times \R^n$ with $\R^n,$  we have
    $$y = \frac{2(x+\lan x,e\ran  e)}{-\lan x+Ax, e\ran}, \quad y' = \frac{2(Ax+\lan Ax,e\ran  e)}{-\lan x+Ax, e\ran}.$$

    It is well-known that if $u \in \R^{1,n}$ with $\lan u, u \ran \neq 0$, then $$F_u(x) = x -2\cdot\frac{\lan x, u \ran}{\lan u, u \ran} \cdot u$$ is a Lorentz transformation.  If we take $u=Ae+e$ and use $\lan Ae, e \ran <0$, then $\lan u, u \ran =2 \lan Ae+e, e \ran \neq 0$. Therefore, using $\lan Ax, Ae+e \ran =\lan Ax+x, e \ran$, 
the following linear map $T$, which is the composition of $A$ with $F_{Ae+e}$, is a Lorentz transformation, 
    \begin{equation}
        Tx = Ax - \frac{\lan Ax+x, e\ran  }{\lan Ae+e, e\ran  }(Ae+e)
    \end{equation}
Note that $$T(e) =  Ae-\frac{\lan Ae+e, e\ran  }{\lan Ae+e, e\ran  }(Ae+e)=  -e.$$ Hence, the restriction of $T$ to the orthogonal complement of $e$, i.e., $\{0\} \times \R^n$, is an isometry in $\Isom(\R^n)$.

Consider the vector
   \begin{equation}
       \alpha = -2\frac{Ae+\lan Ae,e\ran e}{\lan Ae+e, e\ran  }\in\{0\}\times\R^n.
    \end{equation}  Then, we have
    \begin{equation}
        \begin{aligned}
            \frac{1}{2}(y'-Ty-\alpha) =& \frac{Ax+\lan Ax,e\ran  e}{-\lan x+Ax, e\ran} - T\left(\frac{x+\lan x,e\ran  e}{-\lan x+Ax, e\ran}\right) + \frac{Ae+\lan Ae,e\ran e}{\lan Ae+e, e\ran  }\\
            =& -\frac{Ax+\lan Ax,e\ran  e}{\lan x+Ax, e\ran} + \frac{Tx}{\lan x+Ax, e\ran} + \frac{\lan x,e\ran  Te}{\lan x+Ax, e\ran} + \frac{Ae+\lan Ae,e\ran e}{\lan Ae+e, e\ran  }\\
            =& -\frac{Ax+\lan Ax,e\ran  e}{\lan x+Ax, e\ran} + \frac{1}{\lan x+Ax, e\ran}\left(Ax - \frac{\lan Ax+x, e\ran  }{\lan Ae+e, e\ran  }(Ae+e)\right)\\
            &- \frac{\lan x,e\ran  e}{\lan x+Ax, e\ran} + \frac{Ae+\lan Ae,e\ran e}{\lan Ae+e, e\ran  }\\
            =&-\frac{\lan Ax,e\ran  e+\lan x,e\ran  e}{\lan x+Ax, e\ran}-\frac{Ae+e}{\lan Ae+e, e\ran}+ \frac{Ae+\lan Ae,e\ran e}{\lan Ae+e, e\ran  }\\
            =&-\frac{\lan x+Ax, e\ran}{\lan x+Ax, e\ran}e+\frac{\lan Ae+e, e\ran}{\lan Ae+e, e\ran}e=0.
        \end{aligned}
    \end{equation}
    Hence we have $y'=Ty+\alpha$. 
    It follows that $By=Ty+\alpha$ is an isometry of $\R^n$, i.e. $y'=By$.  By definition of $y, y'$, we have $\Phi(x, Ax) =(y, y') =(y, By)$.

    The uniqueness of $B$ follows from the fact that an isometry is determined by its restriction to a set of $n+1$ points which are not contained in a codimension-1 hyperplane in $\R^n$. For example, let $\{v_0, v_1, \cdots, v_n\}$ be $n+1$ vectors in $\mathcal H^n$ such that $v_0 = e$ and $\PP(v_i)$ are independent  in $\R^n$ for $i = 1, \cdots, n$. Then $\Phi(v_i, Av_i) = (u_i, Bu_i)$ with $u_0 = 0$ and $u_1, u_2, ..., u_n$  are linearly independent vectors in $\R^n$. If there is another $B'$ such that $\Phi(v_i, Av_i) = (u_i, B'u_i)$, then $Bu_i = B'u_i$ for $i = 0, \cdots, n$. Hence, $B = B'$.

    Finally, by the symmetry of Pogorelov map, we have
    \begin{equation*}
        \Phi(x,A^{-1}x)=\Phi\left(A(A^{-1}x),A^{-1}x\right)=(By'',y'')=(z,B^{-1}z).\qedhere
    \end{equation*}
\end{proof}

We note that the vector $\alpha$ constructed above has length less than 2. Moreover, the converse of Proposition \ref{prop:AtoB} also holds with this length restriction.

\begin{proposition}\label{prop:BtoA}
    For an isometry $v\mapsto Tv+\alpha$ in $\Isom(\R^n)$ with $T\in \mathrm{O}(n)$ being an orthogonal transformation and $\alpha\in\R^n$ with $\|\alpha\|<2$, there exists a unique $A\in \Isom(\mathcal H^n )$ such that for all $y \in \R^n$ with $(y,Ty+\alpha)\in\Omega$, 
    $\Psi(y, Ty+\alpha) = (x, Ax)$.
    Furthermore, the matrix in $\Isom(\mathcal H^n)$ associated to $v\mapsto T^{-1}v-T^{-1}\alpha$, the inverse  of $v\mapsto Tv+\alpha \in \Isom(\R^n)$, is $A^{-1}$, i.e., for any $y\in\R^n$ such that $(y, T^{-1}y -T^{-1}\alpha)\in\Omega$, 
    $\Psi(y, T^{-1}y -T^{-1}\alpha) =(x', A^{-1}x')$. 
    \end{proposition}

\begin{proof}
    We first verify that the matrix
    \begin{equation}
        A=\frac{1}{4-\|\alpha\|^2}\begin{pmatrix}
            4+\|\alpha\|^2 & 4\alpha^T T\\
            4\alpha & 2\alpha\alpha^T T+(4-\|\alpha\|^2)T
        \end{pmatrix}
    \end{equation}
    is a Lorentz transformation satisfying the requirement. 
    Indeed, let $J$ be the diagonal matrix $\diag(-1,1,\dots,1)$. We have
    \begin{equation}
    \begin{aligned}
    &\ (4-\|\alpha\|^2)^2 A J A^T\\
    =&\begin{pmatrix}
        -(4+\|\alpha\|^2) & 4\alpha^T T \\
        -4\alpha & 2\alpha\alpha^T T+(4-\|\alpha\|^2)T
    \end{pmatrix}
    \begin{pmatrix}
        4+\|\alpha\|^2 & 4\alpha^T \\
        4T^T\alpha & 2T^T\alpha\alpha^T +(4-\|\alpha\|^2)T^T
    \end{pmatrix} \\
    =&\begin{pmatrix}
        M_{11} & M_{12} \\
        M_{21} & M_{22}
    \end{pmatrix} 
    = (4-\|\alpha\|^2)^2 J,
    \end{aligned}
    \end{equation}
    where
    \begin{equation}
    \begin{aligned}
    &M_{11} = -(4+\|\alpha\|^2)^2 + 16\|\alpha\|^2, \\
    &M_{12} = -4(4+\|\alpha\|^2)\alpha^T + 4\alpha^T T\big(2T^T\alpha\alpha^T +(4-\|\alpha\|^2)T^T\big), \\
    &M_{21} = -4\alpha(4+\|\alpha\|^2) + \big(2\alpha\alpha^T T+(4-\|\alpha\|^2)T\big)4T^T\alpha, \\
    &M_{22} = -16\alpha\alpha^T + \big(2\alpha\alpha^T T+(4-\|\alpha\|^2)T\big)\big(2T^T\alpha\alpha^T +(4-\|\alpha\|^2)T^T\big).
    \end{aligned}
    \end{equation}
    Moreover, let $f = f(y, Ty+\alpha)$; then the computation in Proposition (\ref{prop:Pogorelov-inverse}) implies that
    \begin{equation}
        \begin{aligned}
            \Psi(y, Ty+\alpha) &= \frac{1}{\sqrt{f}}\left(\left(||y||^2 - ||Ty+\alpha||^2 + 4, 4y\right), \left( ||Ty+\alpha||^2 - ||y||^2 + 4, 4(Ty+\alpha\right)\right)\\
            & =\frac{1}{\sqrt{f}}\left(\left(4 - ||\alpha||^2 - 2\alpha^TTy, 4y\right),\left(4 + ||\alpha||^2 + 2\alpha^TTy, 4(Ty+\alpha)\right)\right).
        \end{aligned}
    \end{equation}
    Let $x = \frac{1}{\sqrt{f}}(4 - ||\alpha||^2 - 2\alpha^TTy, 4y)$, then
    \begin{align}
        Ax &= \frac{1}{\left(4-\|\alpha\|^2\right)\sqrt{f}}\begin{pmatrix}
            4+\|\alpha\|^2 & 4\alpha^T T\\
            4\alpha & 2\alpha\alpha^T T+(4-\|\alpha\|^2)T
        \end{pmatrix}\begin{pmatrix}
            4 - ||\alpha||^2 - 2\alpha^TTy \\
            4y
        \end{pmatrix}\\
        &= \frac{1}{\sqrt{f}}\begin{pmatrix}
            4 + ||\alpha||^2 + 2\alpha^TTy \\
            4(Ty +\alpha)
        \end{pmatrix}.
    \end{align}
    Therefore, we have $\Psi(y, Ty+\alpha) = \left(x, Ax\right)$.
    Since the image of $\Psi$ is $\HH^n\times\HH^n$, we know $A$ preserves the chosen sheet of the hyperboloid model. Hence, $A\in\Isom(\HH^n)$.
        
    The uniqueness of $A$ follows from the same idea in the proof of the uniqueness part in Proposition \ref{prop:AtoB}.
    Finally, for any $y\in\R^n$ such that $(y, T^{-1}y -T^{-1}\alpha)\in\Omega$, let $y'=T^{-1}y -T^{-1}\alpha$; then $y=Ty'+\alpha$. By the symmetry of $\Psi$, there exists $x''\in\HH^n$, such that
    \begin{equation*}
        \Psi(y, T^{-1}y -T^{-1}\alpha) = \Psi(Ty'+\alpha,y') = (Ax'',x'') = (x', A^{-1}x'),
    \end{equation*}
    where $x'=Ax''$.
\end{proof}

\subsection{Star-shapedness and local convexity}

Recall that a subset $X$ of $\H^n$ or $\R^n$ is \it star-shaped \rm with center $v$ if for each $x \in X$, the geodesic segment joining $v$ and $x$ is contained in $X$. Since geodesics in $\mathcal H^n$ are intersections of 2-dimensional linear subspaces in $\R^{1,n}$ with $\mathcal H^n$, the restriction map $\PP(x_0, x_1,...,x_n)=(x_1, ..., x_n): \mathcal H^n \to \R^n$ sends geodesics through $e$ to lines in $\R^n$ through $0$,  and sends geodesic rays from $e$ to rays from $0$.  Another way to prove this is by noting that the central projection $\PP(x_0, ...,x_n)/x_0: \mathcal H^n \to \{x \in \R^n: ||x||<1\}$ is the canonical map from the hyperboloid model to the Klein model of $\H^n$.
In particular, $\PP$ maps a star-shaped set centered at $e$ to a star-shaped set centered at $0$, and $\PP$ induces a bijection between the set of all geodesic rays in $\mathcal H^n$ from $e$ to the set of all rays in $\R^n$ from $0$.  
In terms of explicit formulas, the geodesic ray $\gamma$ in $\mathcal H^n$ from $e=(1, 0,...,0)$ to $x \in \mathcal H^n-\{e\}$ is given by 
\begin{equation}
    \left\{ \frac{(1-t)e + tx}{\sqrt{-\lan (1-t)e+tx, (1-t)e+tx\ran}}: t \geq 0\right\},
\end{equation}
whose image $\PP(\gamma)$ is the ray $\R_{\geq 0} \PP(x):=\{ s \PP(x): s \geq 0\}$ in $\R^n$ from $0$. Recall that $\Phi(x,y)=(P_1(x,y), P_2(x, y))$.

\begin{lemma}\label{prop:star-shaped} Let $\gamma$ and $\gamma'$ be two distinct geodesic rays in $\mathcal H^n$ from $e$. Then,
\begin{enumerate}[label=(\roman*)]
    \item $P_1(\gamma \times \mathcal H^n) \subset \PP(\gamma)$, $P_2( \mathcal H^n \times \gamma) \subset \PP(\gamma)$, 
    and
    \item $\PP(\gamma) \cap \PP(\gamma') =\{0\}$.
\end{enumerate}

\end{lemma}

\begin{proof} \noindent\emph{(i)}  By the symmetry of the Pogorelov map $\Phi(x,y)=\left(P_1(x,y), P_2(x,y)\right)$ in terms of $x, y$, it suffices to prove the statement for $P_1$. Take any $y \in \mathcal H^n$, and $x \in \gamma -\{e\}$. By   definition,
   \begin{equation*}
       P_1(x,y)=\frac{2\PP(x)}{-\lan x+y,e\ran}=\left(-\frac{2}{\lan x+y,e\ran}\right)\PP(x) \in \PP(\gamma),
       \end{equation*} since $\lan x+y, e \ran <0$.

    \smallskip
    \noindent\emph{(ii)} Take $x \in \gamma -\{e\}$ and $y \in \gamma'-\{e\}$. Since $\gamma \neq \gamma'$, either three vectors $e, x, y$ are linearly independent in $\R^{1,n}$ when $\gamma$ and $\gamma'$ are not contained in a geodesic, or $\PP(x)=-k\PP(y)$ for some $k \in \R_{>0}$ when $\gamma$ and $\gamma'$ are contained in a geodesic. It follows that either $\PP(x)$ and $\PP(y)$ are linearly independent in $\R^n$, or $\PP(x) =-k\PP(y)$. In both cases, the two rays $\PP(\gamma)$ and $\PP(\gamma')$ intersect only at $0$. 
\end{proof}

We say that a subset $\Sigma$ of $\mathcal H ^n$ or $\R^n$ is \emph{radial with respect to $v$}, if each geodesic ray $\gamma$ from $v$ intersects $\Sigma$ in at most one point.

\begin{lemma}[\cite{Pogorelov}]\label{lem:star-shaped}
    Suppose $\Sigma \subset\mathcal H^n$  is a radial set with respect to $e=(1,0,..,0)$ and $e \notin \Sigma$, and $f \from \Sigma \to \mathcal H^n$ is a continuous map. Let $S=\{P_1(x,f(x)): x\in \Sigma\}$. Then 
    \begin{enumerate}[label=(\roman*)]
        \item $S$ is a radial set with respect to  $0$ in $\R^n$ with $0\notin S$, and
        \item the map $ P_1(x, f(x)): \Sigma \to S$  is bijective and continuous. 
    \end{enumerate}
\end{lemma}

\begin{proof} \noindent\emph{(i)} Since  $P_1(x,y)=-2\PP(x)/\lan x+y,e\ran$,  $\PP^{-1}(0) = \{(t, 0,...,0)\in \R^{1,n}: t \in \R\}$ and $e \notin \Sigma$, we see that $0 \notin S$.  Now if $S$ is not radial with respect to $0$, by definition, there are two distinct points $x, y \in \Sigma$ such that $P_1(x, f(x))$ and $P_1(y, f(y))$ are in the same ray in $\R^n$ from $0$. Since $\Sigma$ is radial with respect to $e$, there are two distinct geodesic rays $\gamma$ and $\gamma'$ from $e$ such that $x \in\gamma$ and $y \in \gamma'$.  By Lemma \ref{prop:star-shaped}, we see that $P_1(\gamma \times \mathcal H^n)$ and $P_2(\gamma' \times \mathcal H^n)$ are in different rays in $\R^n$ from $0$. This contradicts that $P_1(x, f(x))$ and $P_1(y, f(y))$ are in the same ray. Therefore, $S$ is radial with respect to $0$.

\smallskip
\noindent\emph{(ii)} It is clear from Definition \ref{P_i} of $P_1$ and continuity of $f$ that $P_1(x, f(x))$ is continuous.  The injectivity of $P_1(x, f(x)): \Sigma  \to S$ follows from that argument above which shows $P_1(x, f(x))  \neq P_1(y, f(y))$ for $x \neq y$. 
\end{proof}

One of the important properties of the Pogorelov map is that it preserves the local convexity of surfaces.  Pogorelov showed that, in the  $3$-sphere $\mathbb S^3$, the map $\Phi$ preserves local convex radial surfaces containing $e$ in their convex sides.  
Pogorelov proved Theorem \ref{lem:local-convex} for smooth locally convex surfaces on pages 302-308 and general locally convex surfaces on pages 314-322 of \cite{Pogorelov}. See also Proposition 26, page 351 in Spivak's book~\cite{Spivak} for a proof for smooth surfaces.  
The proof presented in \cite{Pogorelov} works for the hyperbolic case, so we only summarize his statement below and skip the proof.

\begin{theorem}[Pogorelov] \label{lem:local-convex}
Suppose $X_1$ and $X_2$ are two convex 3-dimensional sets in $\H^3$ that contain $e$ in their interiors. Let $\Sigma_i$ be a surface contained in $\partial X_i$  and $f: \Sigma_1 \to \Sigma_2$ be an isometry.  Then 
\begin{enumerate}[label=(\roman*)]
    \item $S_1 =\{P_1(x, f(x)): x \in \Sigma_1\}$ and 
    $S_2 =\{P_2(x, f(x)): x \in \Sigma_1\}$ are two radial surfaces in $\R^3$ with respect to $0$.
    \item (local convexity) For $i=1,2$ and each point $p \in S_i$, there exists a convex body $Y_p$ containing $0$ such that $S_i \cap \partial Y_p$ contains a neighborhood of $p$ in $S_i$.
    \item There exists a local isometry $g: S_1 \to S_2$ such that $\Phi(x, f(x)) =(y, g(y))$ for all $x \in \Sigma_1$.
\end{enumerate}
\end{theorem}
Note that since $X_i$ contains $e$ in its interior and $\Sigma_i \subset \partial X_i$, $\Sigma_i$ is radial with respect to $e$. 
Let $h_i(x) =P_i(x, f(x)): \Sigma_i \to S_i$ be the homeomorphism produced from Lemma \ref{lem:star-shaped} for $i=1,2$. Then the local isometry $g$ in Theorem \ref{lem:local-convex} (iii) is the homeomorphism $h_2 \circ f \circ h_1^{-1}.$  By Lemma \ref{local+homeo=global}, we have

\begin{corollary}\label{c:isometry} The map $g: S_1 \to S_2$ in Theorem \ref{lem:local-convex} is an isometry.   \end{corollary}

\smallskip

\section{Two extension results}\label{sec:extension}

In this section, we prepare several results that will be used in the proof of Theorem \ref{main2}. 
Recall that a function $f$ defined on an open set $U$ in $\R^n$ is \textit{locally convex} if for each $p\in U$, there exists a convex neighborhood $N_p$ of $p$ in $U$ such that $f|_{N_p}$ is convex.  A well-known lemma about local convexity is the following. We omit the proof. 

\begin{lemma}\label{local2gloab} 
    If $U \subset \R^n$ is a convex open set and $f: U \to \R$ is locally convex, then $f$ is a convex function on $U$. 
\end{lemma}

We now recall a theorem from \cite{TT10}, which ensures the extension of a locally convex function defined outside a closed set of vanishing Hausdorff measure. For any positive integer $m$, we denote by $\mathfrak H^m$ the $m$-dimensional Hausdorff measure in a metric space.

\begin{theorem}[{\cite[Corollary 4.2]{TT10}}]\label{thm:Tabor}
    Let $U$ be an open subset of $\R^n$ and $A$ be a closed subset of $U$ such that $\mathfrak H^{n-1}(A)=0$. If $f\from U- A\to\R$ is locally convex, then $f$ can be extended to a locally convex function on $U$. 
\end{theorem}

The next lemma is an elementary result about Hausdorff measures.

\begin{lemma}\label{lem:Hausdorff}
    Let $R\from \R^3-\{ 0\} \to \mathbb S^2$ be the projection given by $R(x)=x/\|x\|$. Suppose $A\subset \mathbb S^2$ such that $\mathfrak H^1(A)=0$. Then $     \mathfrak H^2(R^{-1}(A))=0$.
\end{lemma}

\begin{proof}
    Let $B=R^{-1}(A)$ and $B_k=B\cap\{x\in\R^3:1/k\leq\|x\|\leq k\}$, for $k\in\Z_+$. Then $B=\cup_{k=1}^\infty B_k$, and it suffices to show that $\mathfrak H^2(B_k)=0$.

    For $k\in\Z_+$, let $I_k=[1/k,k]$. Then $h\from B_k\to A\times I_k$ sending $x$ to $(x/\|x\|,\|x\|)$ is a diffeomorphism. In particular, $h$ is bi-Lipschitz with respect to the product metric on $A\times I_k\subset \mathbb S^2\times I_k$. By $\mathfrak H^1(A)=0$ and $\mathfrak H^1(I_k)<\infty$, we have $\mathfrak H^2(B_k)=0$.
\end{proof}

Another result we need is the following: the path metric of a convex surface does not change if a set of vanishing 1-dimensional Hausdorff measure is removed.  
Recall that a complete convex surface in $\R^3$ is a convex surface whose intrinsic path metric is complete. 

\begin{theorem}\label{thm:path metric}
    Suppose $S\subset \R^3$ is a complete convex surface and $A\subset S$ is a compact subset such that $\mathfrak H^1(A)=0$. Let $d_S$ and $d_{S-A}$ be the path metrics on $S$ and $S-A$, respectively. Then for any $x,y\in S-A$,
    \begin{equation}
        d_S(x,y)=d_{S-A}(x,y).
    \end{equation}
\end{theorem}

\begin{proof} 
    By definition, it is clear that $d_S(x,y) \leq d_{S-A}(x,y)$. Therefore, it suffices to show that for any $\epsilon >0$ and $x,y\in S-A$, we can construct a path $\tilde{\gamma}$ in $S-A$ joining $x$ and $y$ such that the length of $\tilde{\gamma}$ is at most $L+\epsilon$, where $L=d_S(x,y)$.
    Let $\gamma\from [0,1]\to S$ be a shortest geodesic with $\gamma(0)=x$ and $\gamma(1)=y$. Note that $\gamma$ is injective due to the shortest distance property.

    Since $A$ is compact, there exists $\delta_1>0$ such that 
    \begin{equation}
        N_{\delta_1}(A)=\{p\in S:d_S(p,A)\leq\delta_1\}
    \end{equation}  
    is compact in $S$ and does not contain $x$ or $y$.
    Then there exists $\delta_2>0$ such that for any  
     $p \in N_{\delta_1}(A)$   and $\delta\in(0,\delta_2]$, 
     \begin{enumerate}[label=(\roman*)]
         \item the open ball $B_p(\delta):=\{q\in S:d_S(p,q)<\delta\}$ is homeomorphic to an open disk;
         \item $B_p(\delta)$ does not contain $x$ or $y$; and
         \item the set $\partial \hat B_p(\delta):=\{q\in S:d_S(p,q)=\delta\}$ is a simple closed loop in $S$.
     \end{enumerate}
    See \cite{BGP} for details. Let $\hat{B}_p(\delta):=\{q\in S:d_S(p,q)\leq\delta\}$
    be the closed ball.

    For any $\delta\in(0,\delta_2]$, by $\mathfrak H^1(A)=0$, there exists $p_n \in S$  and $r_n>0$, $n=1,2,\cdots,$ such that $A\subset \bigcup_{n=1}^\infty B_{p_n}(r_n)$ and $\sum_{n=1}^\infty r_n<\delta$. 
    By the compactness of $A$,  we find finitely many balls $B_{p_1}(r_1), ..., B_{p_m}(r_m)$ covering $A$.
   Note that, by construction,  $x,y\notin B_{p_n}(r_n)$, for $1\leq n\leq m$.

    For each $p_n$, consider a map $f_n\from A\to \R$ defined by $f_n(q)=d_S(p_n,q)$. By the triangle inequality, $f_n$ is Lipschitz. In particular, it implies that $\mathfrak H^1(f_n(A))=0$. Hence, for almost all $r \in \R_{>0}$, $\partial \hat B_{p_n}(r)$ contains no points in $A$. For $n=1,2,...,m$,  we increase $r_n$ a little such that $\partial \hat B_{p_n}(r_n)$ is disjoint from  $A$, and $\sum_{n=1}^\infty r_n<\delta$ still holds.

    Now we inductively modify the curve $\gamma$ to make it disjoint from $A$ as follows. For $1\leq n\leq m$, let $\gamma_n:[0,1] \to S$ be the inductively constructed curve. 
    Set $\gamma_0=\gamma$, and assume that we have constructed $\gamma_{n-1}$. In the $n$-th step, if $\gamma_{n-1}\cap\hat{B}_{p_n}(r_n)=\varnothing$, define $\gamma_n=\gamma_{n-1}$. 
    Otherwise, let 
    \begin{equation}
        a_{n-1}=\min\{t\in[0,1]:\gamma_{n-1}(t)\in \hat{B}_{p_n}(r_n)\}
    \end{equation}
    and
    \begin{equation}
        b_{n-1}=\max\{t\in[0,1]:\gamma_{n-1}(t)\in \hat{B}_{p_n}(r_n)\},
    \end{equation}
    and $\gamma_n$ is obtained from $\gamma_{n-1}$ by retaining the portions $\gamma_{n-1}([0,a_{n-1}])$,  $\gamma_{n-1}([b_{n-1}, 1])$, and replacing  $\gamma_{n-1}([a_{n-1},b_{n-1}])$ by a subarc in $\partial \hat B_{p_n}(r_n)$ that joins $\gamma_{n-1}(a_{n-1})$ and $\gamma_{n-1}(b_{n-1})$.

    We note that for any $z\in A\cap\gamma$, let $i$ be the minimal index such that $z\in B_{p_i}(r_i)$, then $z\in \gamma_{i-1}$ but $z\notin \gamma_i$. Furthermore, in each step, the added arc on $\partial \hat B_{p_n}(r_n)$ is disjoint from $A$. Therefore, $\gamma_m\cap A=\varnothing$, i.e., $\gamma_m$ is a path in $S-A$ joining $x$ to $y$. We claim that $l(\partial \hat B_p(\delta)) \leq 2\pi \delta$ for any ball $\hat B_p(\delta)$. Assuming this, since the length of the added subarc is no greater than $l(\partial \hat B_{p_n}(r_n))$, we have
    \begin{equation}
        l(\gamma_m)\leq l(\gamma)+\sum\limits_{n=1}^m l(\partial \hat B_{p_n}(r_n))\leq L+2\pi\sum\limits_{n=1}^m r_n\leq L+ 2\pi \delta.
    \end{equation}
    Then the result follows by taking $\delta =\epsilon/2\pi$. 
    
    It remains to show  $l(\partial \hat B_p(\delta)) \leq 2\pi \delta$ on $S$. 
Let $\beta:[0, 1] \to \partial \hat B_p(\delta)$ be a parametrization of  $\partial \hat B_p(\delta)$. If $l = l(\partial \hat B_p(\delta)) > 2\pi \delta$, there exists a partition $0 = t_0<t_1<...<t_n=1$ such that $q_i = \beta(t_i)$ with $q_0 = q_n$ and $\sum_{i=0}^{n-1}d_S(q_i,q_{i+1})\geq  \pi\delta+ l/2.$
Let $\alpha_i$ be the angles of geodesic triangles $\triangle pq_iq_{i+1}$ at $p$. Since $S$ is non-negatively curved, $\sum_{i=0}^{n-1}\alpha_i \leq 2\pi$. Consider the Euclidean comparison triangles  $\triangle p'p'_ip'_{i+1}$ with the same edge lengths as that of  $\triangle pq_iq_{i+1}$ and  angles $\alpha_i'$ at $p'$. By the Toponogov comparison theorem for Alexandrov spaces with curvature bounded below \cite{BBI}, we have $\alpha_i\geq \alpha_i'$. Hence, $2\pi\geq \sum_{i=0}^{n-1} \alpha_i\geq \sum_{i=0}^{n-1} \alpha_i'$. However, in the Euclidean triangle $\triangle p'p'_ip'_{i+1}$,   $l([p',p_i']) = \delta$, and $l([p_i'p_{i+1}'])=2 \delta \sin(\alpha_i'/2) \leq \delta \alpha_i'$. Then we have a contradiction  
\begin{equation*}
    2\pi\delta<\pi\delta + \frac{l}{2}\leq \sum_{i=0}^{n-1}d_S(p_i,p_{i+1}) = \sum_{i=0}^{n-1}  l([p_i',p_{i+1}']) \leq \delta\sum_{i=0}^{n-1}   \alpha_i' \leq 2\pi\delta.\qedhere
\end{equation*}
\end{proof}

\begin{remark}
    The above argument also shows that $S-A$ is path-connected.
\end{remark}

\smallskip

\section{Proof to Theorem \ref{main2}}\label{sec:main-proof}

The proof splits into three cases:
(i) both $X_1$ and $X_2$ are 3-dimensional, (ii) both $X_1$ and $X_2$ are 2-dimensional, and (iii) exactly one of $X_1$ and $X_2$ is 3-dimensional. We first apply the Pogorelov map $\Phi$ to prove Theorem \ref{main2} for the case (i) in Section \ref{subsec:3d-3d}.  The other two cases are treated in Sections \ref{subsec:2d-2d} and \ref{subsec:2d-3d}.

\subsection{Proof of Theorem \ref{main2} in case (i): both $X_1$ and $X_2$ are $3$-dimensional}\label{subsec:3d-3d}

By composing with isometries of $\mathcal H^3$, we may assume that $e$ is in the interior of each $X_i$. By the convexity of $X_i$, we know that the $\partial X_i$ are convex surfaces which are radial with respect to  $e=(1,0,0,0)$. Let $f:\partial X_1 \to \partial X_2$ be an isometry and $S_i=\{P_i(x, f(x)):x\in\partial X_1\}$, $i=1,2$,  be the image in $\R^3$. Then, by Lemma \ref{lem:star-shaped} and Theorem \ref{lem:local-convex},  $S_i$ are locally convex surfaces that are radial with respect to  $0$ in $\R^3$.  Let $g: S_1 \to S_2$ be the local isometry produced from Theorem \ref{lem:local-convex} such that $\Phi(x, f(x)) =(y, g(y))$ for all $x \in \partial X_1$. By Corollary \ref{c:isometry}, we know that $g$ is an isometry. Our goal is to extend each $S_i$ to a closed convex surface $\tilde S_i$ in $\R^3$ and extend the isometry $g$ to an isometry $\tilde g$ between $\tilde S_1$ and $\tilde S_2$.

For $i=1,2$, let $W_i =\{ x/||x||: x \in S_i\}$ and $Z_i=\{k x: k\geq 0, x \in W_i\}=\{k x: k\geq 0, x \in S_i\}$. Clearly,  $W_i\subset \mathbb S^2$, $Z_i$ is star-shaped with center $0$.  Since $S_i$ is a radial surface without boundary, $Z_i -\{0\}$ is homeomorphic to $S_i \times \R $ and is open in $\R^3$. 

\begin{lemma} 
\label{2dHauss}
The 1-dimensional Hausdorff measure $\mathfrak H^1(\mathbb S^2 -W_i)$  of $\mathbb S^2-W_i$ is zero.  Furthermore, the 2-dimensional Hausdorff measure $\mathfrak H^2(\R^3-Z_i)$ of $\R^3-Z_i$ is zero.  
\end{lemma}

\begin{proof} The last statement in the lemma follows from the first by using Lemma \ref{lem:Hausdorff}.

For the first statement, note that the map $\Pi_P: \mathcal H^3 \to \H^3_P:=\{ y \in \R^3: ||y||<1\}$ given by $\Pi_P(x) =\PP(x)/(1+x_0)$ sending $(x_0, x_1, x_2, x_3)$ to $(x_1, x_2, x_3)/(1+x_0)$ is the canonical isometry from the hyperboloid model to the Poincar\'e model $\H^3_P$ of $\H^3$. It follows that $M_i =\Pi_P(X_i)$ is a closed convex set in $\H^3_P$ containing $0$ in its interior.  By our assumption, $\mathfrak H^1(\overline{M_i} \cap \mathbb S^2)=0$. Here 
$\mathbb {S} ^2= \partial \H^3_P$ is the conformal boundary at infinity of the hyperbolic 3-space and $\overline{L}$ is the closure of a set $L$ in $\R^3$.  Thus, the lemma follows if we prove that
\begin{equation} \label{eq1234} \overline{M_i} \cap \mathbb S^2 = \mathbb S^2 -W_i. \end{equation}
Consider the surjective map $h_i(x) =P_i(x, f(x)): \partial X_1 \to S_i$. By the definition of $P_i$ in (\ref{P_i}),  $$\frac{h_i(x)}{||h_i(x)||} =\frac{\PP(x)}{||\PP(x)||}=\frac{\Pi_P(x)}{||\Pi_P(x)||}. $$  In particular, 
$$W_i =\left\{ \frac{y}{||y||}: y \in S_i\right\} = \left\{ \frac{h_i(x)}{||h_i(x)||}: x \in \partial X_1\right\} =\left\{\frac{\Pi_P(x)}{||\Pi_P(x)||}: x \in \partial X_1\right\} =\left\{\frac{z}{||z||}: z \in \partial M_i\right\}.$$
Let $R(z)={z}/{||z||}: \R^3 -\{0\} \to \mathbb S^2$ be the radial projection. Then $W_i =R(\partial M_i)$. Therefore, \eqref{eq1234} is a consequence of the following fact about convex geometry. That is, if $M$ is a convex set contained in the open unit ball $\mathbb B^3$ such that $M$ is closed in $\mathbb B^3$ and $0$ is in the interior of $M$, then $\overline{M} \cap \mathbb S^2 =\mathbb S^2-R(\partial M)$. It is well-known that the compact set $\overline{M}$ is convex, its interior $\inte(\overline{M}) =\inte(M)$, and $\partial (\overline{M}) \cap \mathbb B^3 =\partial M$. 
Now, if $y =R(x) \in R(\partial M)$ with $x \in \partial M$, then since $0\in \inte(\overline{M})$, the line segment $[0, x] = \gamma \cap M$ where $\gamma$ is the ray from $0$ containing $x$. But $y \notin [0,x]$ and hence $y \notin \overline{M}$. It follows that  $R(\partial M) \subset \mathbb S^2 -\overline{M}$. On the other hand, if $y \in \mathbb S^2 -\overline{M}$, consider $[0, y] \cap \overline{M}$. Since $0 \in \inte(\overline{M})$, there exists $x \in \partial M$ such that $x \in [0, y]$. This implies $y=R(x) \in R(\partial M)$, i.e., $\mathbb S^2 -\overline{M} \subset R(\partial M)$. 
\end{proof}

    Define the function $q_i\from Z_i\to \R$ by
    \begin{equation}
        q_i(\alpha)=\inf \{k>0:  {\alpha}/{k} \in S_i\}\geq 0,
    \end{equation}
    which has the following properties:
    \begin{enumerate}[label=(\roman*)]
        \item $q_i(k\alpha)=k \cdot q_i(\alpha)$, for any $\alpha\in Z_i$ and $k>0$;
          \item $Y_i=\{\alpha:q_i(\alpha)\leq 1\} =\{ k x:  k \in[0,1], x \in S_i\}$ and $S_i =\{ x \in Z_i:  q_i(x)=1\}$;
        \item $q_i$ is a locally convex function on the open set $Z_i -\{0\}$.  
    \end{enumerate}

    Property (i) follows from the definition, and Property (ii) follows from the definition and the fact that $S_i$ is radial with respect to $0$.  To prove property (iii), take any $a \in Z_i -\{0\}$. We will show $q_i$ is convex in a neighborhood of $a$. 
    Let $a^*$ be the point in $S_i$ such that $a^*= ta$ for some $t> 0$. By Theorem \ref{lem:local-convex},  there exists a convex body $X \subset \mathbb{R}^3$ such that $0 \in X$ and $\partial X\cap S_i$ contains a neighborhood of $a^*$ in $S_i$. Since $\partial X$ is a surface, we can choose a sufficiently small round ball $B \subset \mathbb{R}^3-\{0\}$ centered at $a^*$ such that $S_i \cap C(B)= \partial X \cap C(B)$, where $C(B) =\{ kx: k >0, x \in B\}$ is the open cone over $B$ from $0$. Since $B$ is an open convex set,  $C(B)$ is an open convex set and contains the point $a$. We claim that $C(B) \cap Y_i$ is convex. 
    Indeed, since $Y_i$ and $X$ are star-shaped with respect to $0$, we have $C(B)\cap Y_i  =\{ kx: 0<k \leq 1, \quad  x \in C(B) \cap S_i\} =\{ k x: 0<k \leq 1, \quad  x \in C(B) \cap \partial X\} =C(B)  \cap X$ which is the intersection of two convex sets. 
    To show that $q_i$ is a convex function on $C(B)$,  by property (i) that $q_i$ is positive 1-homogeneous, it suffices to verify that for any $x, y\in C(B)$, 
    $$q_i\bigl(\frac{x+y}{2}\bigr)\leq \frac{1}{2}\bigl(q_i(x) + q_i(y)\bigr),$$
    or equivalently, 
    $$q_i(x+y)\leq q_i(x) + q_i(y).$$ For any $\epsilon>0$, we can choose $\lambda$ and $\mu$ respectively such that  ${x}/{\lambda}, {y}/{\mu}\in C(B) \cap Y_i$ and 
    $$\lambda \leq q_i(x) + \epsilon, \quad \mu \leq q_i(y) + \epsilon.$$
    Since $C(B) \cap Y_i $ is convex, the point 
    $$\frac{x+y}{\lambda+\mu} = \frac{\lambda}{\lambda+\mu}\cdot\frac{x}{\lambda} + \frac{\mu}{\lambda+\mu}\cdot\frac{y}{\mu}\in C(B)\cap Y_i.$$
    Therefore, 
    $$q_i(x+y) \leq \lambda+\mu \leq q_i(x) + q_i(y) + 2\epsilon.$$
    By taking $\epsilon \to 0$, we obtain $q_i(x+y)\leq q_i(x) + q_i(y)$. 

    Applying Lemma \ref{2dHauss} and Theorem \ref{thm:Tabor} to $U = \R^3$ and the closed subset $A = (\R^3-Z_i)\cup \{0\}$, we obtain that $q_i$ has a unique locally convex extension $\tilde{q}_i$ to $\R^3$. By Lemma \ref{local2gloab}, $\tilde{q_i}: \R^3 \to \R$ is a convex function. Moreover, since $Z_i$ is the union of rays starting at $0$ and $Z_i$ is dense in $\R^3$, we have $\tilde{q}_i(k\alpha)=k \cdot \tilde q_i(\alpha)$, for any $\alpha\in \R^3$ and $k>0$. 
    
    Let $\tilde{Y}_i=\{\alpha:\tilde{q}_i(\alpha)\leq 1\}$ and $\tilde{S}_i=\{\alpha:\tilde{q}_i(\alpha)=1\}$. We claim that $\tilde{Y}_i$ is a compact convex set containing $0$.  
    Indeed, $\tilde{Y}_i$ is closed and convex since $\tilde{q}_i$ is continuous and convex.  If $\tilde Y_i$ is not compact, there exist $y_n \in \tilde Y_i$ such that $|y_n| \to \infty$. Without loss of generality, we may assume that $y_n/|y_n|\to b \in \mathbb S^2$. Then $\tilde q_i(b) = \lim_{n\to\infty} \tilde q_i(y_n)/|y_n| = 0$ since $\tilde q_i(y_n)\leq 1$. However, since $Z_i-0$ is dense in $\R^3$, there exist a sequence of points $s_n\in S_i$ and a sequence of positive numbers $k_n$ such that $k_ns_n\to b$. Since $\tilde q_i(s_n) = 1$ by definition, we have $k_n\to 0$. Notice that by the inequality (\ref{boundness}) of the Pogorelov map, $|s_n|\leq 2$, then $|b| =\lim_{n\to\infty}  |k_ns_n| = 0$, which contradicts the fact that $b \neq 0$. Therefore, $\tilde{Y}_i$ is a convex body and $\tilde{S}_i=\partial\tilde{Y}_i$ is a closed convex surface in $\R^3$.

    By a theorem of Vrećica~\cite[Theorem 1]{vrecica}, the projection  $R\from \tilde{S}_i\to \mathbb{S}^2$ given by $R(\alpha)=\alpha/\|\alpha\|$ is a co-Lipschitz homeomorphism, i.e., $R^{-1}: \mathbb{S}^2 \to \tilde{S}_i$ is Lipschitz. 
    Then by $\mathfrak H^1(\overline{M}_i\cap\partial\H^3_P)=0$, we see $\tilde{S}_i-S_i=R^{-1}(\overline{M}_i\cap\partial\H^3_P)$ has 1-dimensional Hausdorff measure 0. By Theorem \ref{thm:path metric}, $d_{\tilde{S}_i}|_{S_i}=d_{S_i}$.

    We claim that the isometry $g\from S_1\to S_2$  extends to an isometry $\tilde{g}\from\tilde{S}_1\to\tilde{S}_2$. To see this, let $\alpha\in\tilde{S}_1-S_1$. By $\mathfrak H^1(\tilde{S}_1-S_1)=0$, we know $\tilde{S}_1-S_1$ has no interior points. Thus we can find a sequence $\{\alpha_n\}\subset S_1$ such that $\alpha_n\to \alpha$ in $\tilde{S}_1$ as $n\to\infty$. Since  $\tilde{S}_2$ is compact, there exists $\beta\in\tilde{S}_2$ such that $g(\alpha_n)\to \beta$ after passing to a subsequence. Since $g: S_1 \to S_2$ is a homeomorphism, we see that $\beta \notin S_2$, i.e., $\beta\in \tilde{S}_2-S_2$. 
    Moreover, such $\beta$ is unique. Indeed, if there is a second sequence $\{\alpha_n'\}\subset S_1$ such that $\alpha_n '\to \alpha$ and $g(\alpha_n')\to \beta'$,  then, by Theorem \ref{thm:path metric},  we have
    \begin{equation}
        d_{\tilde{S}_1}(\alpha_n,\alpha '_n)=d_{S_1}(\alpha_n,\alpha '_n)=d_{S_2}(g(\alpha_n),g(\alpha '_n))=d_{\tilde{S}_2}(g(\alpha_n),g(\alpha '_n))\to d_{\tilde{S}_2}(\beta,\beta').
    \end{equation}
    Therefore $d_{\tilde{S}_2}(\beta,\beta')=0$, i.e., $\beta=\beta'$. Defining $\tilde{g}(\alpha)=\beta$ for $\alpha\in\tilde{S}_1-S_1$ as above,  we obtain a map $\tilde{g}\from\tilde{S}_1\to\tilde{S}_2$. By applying the same argument to the isometry $g^{-1}\from S_2\to S_1$, we obtain an extension $\widetilde{g^{-1}}$ from $\tilde{S}_2$ to $\tilde{S}_1$. It follows from the definition, $\widetilde{g^{-1}}$ is the inverse of $\tilde{g}$. This implies that $\tilde{g}$ is a bijection.  

    We claim that the extended map $\tilde{g}$ is an isometry. Indeed, for any $\alpha,\alpha'\in\tilde{S}_1$, let $\{\alpha_n\},\{\alpha'_n\}$ be two sequences in $S_1$ such that $\alpha_n\to \alpha$ and $\alpha'_n\to \alpha'$ as $n\to\infty$. Then, by the above discussion, $g(\alpha_n)\to\tilde{g}(\alpha)$ and $g(\alpha'_n)\to\tilde{g}(\alpha')$ in $\tilde{S}_2$. Therefore, by Theorem \ref{thm:path metric} again, we have
    \begin{equation}
        d_{\tilde{S}_1}(\alpha_n,\alpha'_n)=d_{S_1}(\alpha_n,\alpha'_n)=d_{S_2}(g(\alpha_n),g(\alpha'_n))=d_{\tilde{S}_2}(g(\alpha_n),g(\alpha'_n))\to d_{\tilde{S}_2}(\tilde{g}(\alpha),\tilde{g}(\alpha')).
    \end{equation}
    By taking limit,  we have $d_{\tilde{S}_1}(\alpha_n,\alpha'_n)\to d_{\tilde{S}_1}(\alpha,\alpha')$. Hence $d_{\tilde{S}_1}(\alpha,\alpha')=d_{\tilde{S}_2}(\tilde{g}(\alpha),\tilde{g}(\alpha'))$, and  $\tilde{g}$ is an isometry.

    By Pogorelov's Theorem \ref{pogo1}, there exists an isometry $\phi$ of $\R^3$ such that $\phi|_{\tilde{S}_1}=\tilde{g}$. In particular, $\phi|_{S_1}=g$, and $\phi$ sends the convex body bounded by $\tilde S_1$ to the convex body bounded by $\tilde S_2$, i.e., $\phi(\tilde Y_1) =\tilde Y_2$. Now     $0$ is in the interior of $\tilde{Y}_1$, then $\phi(0)$ is contained in the interior of $\tilde{Y}_2$. By Proposition \ref{prop:Pogorelov-inverse}, for any $\beta\in Y_2$, $\|\beta\|<2$. Thus, for any $\beta\in\tilde{Y}_2$, $\|\beta\|\leq 2$. We then have $\|\phi(0)\|< 2$ since $\phi(0)$ is an interior point in $\tilde{Y}_2$. Therefore, by Proposition \ref{prop:BtoA}, there exists $\tilde{\phi}\in\Isom(\mathcal H^3)$, such that for all $y\in S_1$, 
    $\Psi(y,\phi(y))=(x,\tilde{\phi}(x))$. Note that the condition $(y, \phi(y)) \in \Omega$ is automatically satisfied.

    Now for any $x\in \partial X_1$, let $y=P_1(x,f(x))\in S_1$. By the definition of $g$, we have $\Phi(x, f(x)) =(y, g(y))$.  By Proposition \ref{prop:Pogorelov-inverse},  $\Psi: \Omega \to \mathcal H^3 \times \mathcal H^3$ is the inverse of $\Phi|_{\mathcal H^3 \times \mathcal H^3}$. It follows that 
    \begin{equation}
        (x,f(x))=\Psi(y,g(y))=\Psi(y,\phi(y))=(x,\tilde{\phi}(x)).
    \end{equation}
    Thus $\tilde{\phi}|_{\partial X_1}=f$, and the proof is complete.

\subsection{Proof of Theorem \ref{main2} in case (ii): both $X_1$ and $X_2$ are 2-dimensional}\label{subsec:2d-2d}

Suppose that $X_1$ and $X_2$ are closed convex sets of dimension two in $\H^3_P$. Without loss of generality, we assume the $X_i$ are contained in the canonical copy $\H^2_P\subset\H^3_P$ where $\H^2_P =\H^3_P \cap \{(x,y,0): x, y \in \R\}$.  Take two copies of $X_i$ and glue them along their 1-dimensional boundaries by the canonical map.  We obtain the Alexandrov double $D(X_i)$ of $X_i$. Let the images of these two copies in $D(X_i)$ be $X_i^+$ and $X_i^-$ and their common boundary be $Y_i$, i.e., $Y_i =\partial X_i^{\pm}$. It follows that each $X_i^{\pm}$ is isometric to $X_i$, and  
$D(X_i) = \mathring{X}^+_i \sqcup \mathring{X}^-_i\sqcup Y_i$. For simplicity, we will consider $X^{\pm}_i$ as a convex subset of $\H^2$ in the proof below. In particular, if $p, q \in X^{\pm}_i$, we use $[p,q]$ to be the unique shortest geodesic in $X^{\pm}_i$ from $p$ to $q$. Note that $[p,q]$ is also a geodesic in $\H^2$. 

The key observation for the proof is that if $p \in X^{+}_i$ (or $X^-_i$) and $q \in X^{+}_i-Y_i$ (or $X^{-}_i-Y_i$), then the geodesic segment $[p,q]$ is the unique shortest geodesic path in $D(X_i)$ joining $p$ and $q$. To see this, suppose otherwise that there is a shortest path $\gamma \neq [p,q]$ from $p$ to $q$ in $D(X_i)$. Then $\gamma$ is not contained in $X^{\pm}_i$ since $[p,q]$ is the unique shortest path in the convex set $X^{\pm}_i$. Therefore,  $\gamma -\{p,q\}$ must intersect $Y_i$. Let $Q: D(X_i) \to X_i$ be the map such that the restriction $Q|: X^{\pm}_i \to X_i$ is the identity map.  Then $Q(\gamma)$ is a path in the convex set $X_i$ joining $Q(p)$ and $Q(q)$ and $Q(\gamma) -\{Q(p), Q(q)\}$ intersects $\partial X_i$. Therefore, $Q(\gamma) \neq [Q(p), Q(q)]$ and the length of $Q(\gamma)$ is strictly larger than the length of $[Q(p), Q(q)]$ which is equal to $l([p,q])$.  On the other hand, the length of $Q(\gamma)$ is equal to the length of $\gamma$.  It follows that the length $l(\gamma) =l(Q(\gamma)) > l([p,q])$, which contradicts the choice of $\gamma$.

Given an isometry $f: D(X_1) \to D(X_2)$, our goal is to show that $f(X_1^+)$ is equal to either $X_2^+$ or $X_2^-$, i.e., $f(Y_1) = Y_2$.  Assuming this and composing with the isometric reflection of the double $D(X_i)$ which switches $X^+_i$ with $X^-_i$ if necessary, we may assume that $f(X_1^+) =X_2^+$ and $f(X_1^-) =X_2^-$. Since $X_i$ are two-dimensional, the restrictions of the isometry $f$ to $X_1^+$ and  $X_1^-$ extend to isometries of $\H^2_P$. Since these two isometries agree on $Y_1$, they are identical, i.e., $f$ is induced by a rigid motion of $\H_P^2$. 

It remains to prove that $f(Y_1) = Y_2$. By applying the same argument to $f^{-1}$, we only need to prove $f(Y_1)\subset Y_2$.  Suppose otherwise, without loss of generality, we may assume that $f(Y_1)\cap \mathring{X}_2^+\neq \varnothing$. Since $f$ is an isometry, it follows that $f(\mathring{X}_1^+)\cap \mathring{X}_2^+\neq \varnothing$ and $f(\mathring{X}_1^-)\cap \mathring{X}_2^+\neq \varnothing$.

Note that the non-empty sets $f(X_1^+)\cap X_2^+$ and $f(X_1^-)\cap X_2^+$ are convex. Indeed, if $f(X_1^+)\cap X_2^+$ is not convex, there exist two distinct points $p, q\in  X_1^+$ with $f(p), f(q)\in X_2^+$ such that the geodesic segment $[f(p), f(q)]$ in $X^+_2$ is not contained in $f(X_1^+)\cap X_2^+$. Since $X_1^+$ is convex,  the geodesic segment $[p,q]$ joining $p$ to $q$ is contained in $X_1^+$.  It follows that  $[f(p), f(q)] \not\subset f(X_1^+)\cap X_2^+$ implies $f([p,q]) \neq [f(p), f(q)]$. Hence,  the length of $f([p,q])$ is strictly greater than the length of $[f(p), f(q)]$. 
On the other hand, $f$ is an isometry between the two metric doubled surfaces $(D(x_1), d_{D(X_1)})$ and $(D(X_2), d_{D(X_2)})$ and preserves lengths.  Furthermore, 
both $[p,q]$ and $[f(p), f(q)]$ are the shortest paths in $D(X_i)$ from $p$ to $q$ and $f(p)$ to $f(q)$ respectively. We obtain
$$d_{D(X_1)}(p,q)=l([p,q]) =l(f([p,q])) > l([f(p), f(q)])=d_{D(X_2)}(f(p), f(q)),$$ which contradicts the fact that $f$ is an isometry.
By the same argument, $f(X_1^-)\cap X_2^+$ is also convex. 

Since $f({X}_1^+)\cap {X}_2^+$ and $f({X}_1^-)\cap {X}_2^+$ are two convex sets with disjoint interiors in the convex surface $X_2^+$, their 1-dimensional intersection must be a geodesic, i.e., 
$$\beta^+ = f(Y_1)\cap {X}_2^+ =f({X}_1^+)\cap f({X}_1^-) \cap {X}_2^+ $$  is a geodesic in $X_2^+$. 
Similarly, $$\beta^- = f(Y_1)\cap {X}_2^- =f({X}_1^+)\cap f({X}_1^-) \cap {X}_2^+$$
is either the empty set or a geodesic in  $X_2^-$.  Furthermore, since $\beta^{\pm}$ are geodesic in $\H^2$, $f^{-1}(\beta^{\pm})$ are geodesic in $X_1^{\pm}$ and are contained in $Y_1$. Notice that any geodesic connecting two points in $Y_1$ in $X_1$ is a geodesic in $\H^2$, so $f^{-1}(\beta^{\pm})$ are also geodesics in $\H^2$. See Figure \ref{onering} for the illustration.

\begin{figure}[htbp]
  \centering  
\includegraphics[width=0.9\textwidth]{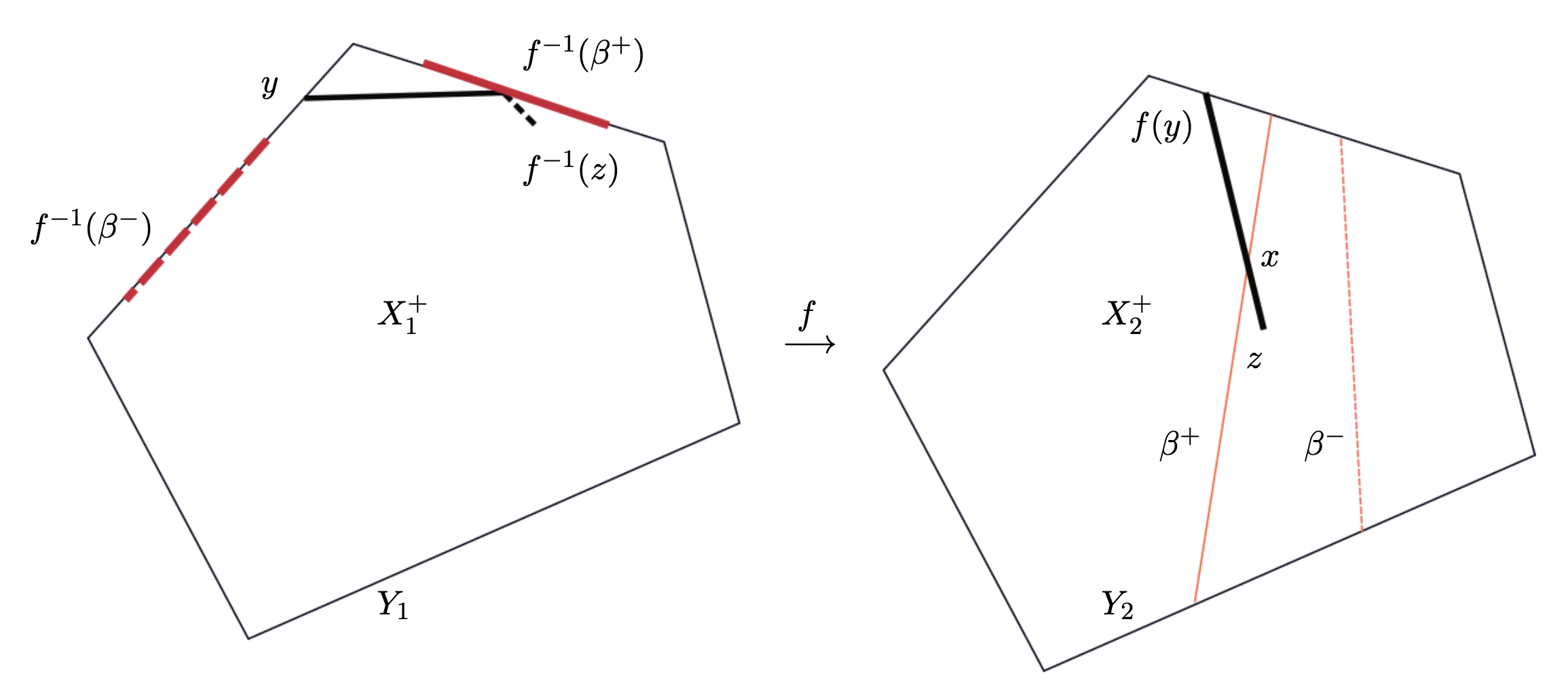}

\caption{Case (ii) where both $X_1$ and $X_2$ are 2-dimensional. Picture drawn in the Klein model.}
\label{onering}
\end{figure}

We claim that $f(Y_1) \neq \beta^+ \cup \beta^-$, i.e., $Y_1 \neq f^{-1}(\beta^+) \cup f^{-1}(\beta^-)$.
Otherwise, $Y_1= f^{-1}(\beta^+) \cup f^{-1}(\beta^-).$  This means that the closed convex set $X_1 \subset \H^2_P$ is bounded by two distinct geodesics $f^{-1}(\beta^+)$ and $f^{-1}(\beta^-)$ (if $\beta^- \neq \varnothing$) or by exactly one geodesic  $f^{-1}(\beta^+)$ (if $\beta^- =\varnothing$). In both cases, since $\dim(X_i)=2>1$, we see that the Hausdorff measure $\mathfrak{H}^1(\overline{X_1} \cap \partial \H^2_P) >0$. This contradicts the assumption.

Now take a point $y\in Y_1- f^{-1}(\beta^+) \cup f^{-1}(\beta^-)  $. Then  $f(y) \notin \mathring X_2^+ \cup \mathring X_2^-$, and hence $f(y) \in Y_2$.  Consider a point $x$ in $\beta^+ \cap \mathring X_2^+$ and the geodesic $\gamma$ in $X_2^+$ containing $f(y)$ and $x$. Since $x$ is in the interior of $X_2^+$, there exists a point $z \in \gamma  -Y_2$  such that the geodesic segment $[f(y), z]$ contains $x$ in its interior, i.e., $[f(y), z]$ intersects both $f(\mathring X_1^+)$ and $f(\mathring X_1^-)$.   Due to $z \in \mathring X_2^+$ and $f(y) \in X_2^+$,  the geodesic segment $[f(y), z]$ is the unique shortest geodesic segment connecting $z$ and $f(y)$ in $D(X_2)$. This shows $f^{-1}([f(y), z])$ is the unique shortest geodesic connecting $y$ and $f^{-1}(z)$ in $D(X_1)$.  Note that the point $f^{-1}(z)$ is in $\mathring X_1^+$ or $\mathring X_1^-$; say $f^{-1}(z) \in \mathring X_1^+$.  Since $y \in Y_1 \subset X_1^+$, we see that the shortest geodesic between $y$ and $f^{-1}(z)$  is the geodesic segment $[y, f^{-1}(z)] \subset X_1^+$. By the uniqueness of the shortest geodesics from $y$ to $f^{-1}(z)$, we have $f^{-1}([f(y), z]) =[y, f^{-1}(z)]$.  In particular, it implies that  $f^{-1}([f(y), z])$ is disjoint from $\mathring X_1^-$. But that contradicts the construction that $[f(y), z]$ intersects $f(\mathring X_1^-)$.  This ends the proof of case (ii).

\subsection{Proof of Theorem \ref{main2} in case (iii):  exactly one of $X_1$ and $X_2$ is 3-dimensional}\label{subsec:2d-3d}
We begin with several lemmas in convex geometry.
Recall that a subset $A$ of $\mathbb S^2$ is \emph{convex}, if for any $x,y\in A$, $A$ contains a shortest-distance geodesic connecting $x$ and $y$. For any $x\in \mathbb S^2$, let $B_x(\pi/2)$ be the open hemisphere centered at $x$.  
\begin{lemma}\label{lem:convex-hemisphere}
    Suppose $A$ is a closed convex subset in $\mathbb{S}^2$ and $A$ does not contain any pair of antipodal points. Then there exists $\alpha\in A$     
    such that $A\subset B_\alpha(\pi/2)$. 
\end{lemma}

\begin{proof}
    Let $K$ be the convex hull of $A$ in $\R^3$ and $C(A)=\{ta:t\geq 0, a\in A\}$ be the cone over $A$ from $0$. Since $A$ is convex in $\mathbb{S}^2$, we know $C(A)$ is convex in $\R^3$ and $K \subset C(A)$.
    
    We first claim that the origin $0 \notin K$. 
    Suppose otherwise that $0\in K$, and then there exist $a_1,a_2,\dots,a_n\in A$ and $\lambda_1,\lambda_2,\dots,\lambda_n>0$ such that 
    $\sum_{i=1}^n \lambda_i a_i=0$ with $\sum_{i=1}^n \lambda_i=1$. It follows that
    $$-a_1=\frac{1}{\lambda_1}\sum\limits_{i=2}^n \lambda_i a_i.$$
    Since $C(A)$ is convex, we have $\sum_{i=2}^n \lambda_i a_i\in C(A)$ and hence $-a_1\in C(A)$. Then by $-a_1\in \mathbb{S}^2$, we have $-a_1\in A$. This is a contradiction of the assumption that $A$ does not contain any pair of antipodal points. 

    Since $ 0 \notin K$, $d_E(0,K)>0$ where $d_E(\cdot, \cdot)$ is the Euclidean distance in $\R^3$. Since $K$ is compact, there exists $\beta\in K$ such that $\|\beta\|=d_E(0,\beta)=d_E(0,K)$. Let $\alpha=\beta/\|\beta\|\in \mathbb{S}^2$. By $\beta\in K\subset C(A)$, we have $\alpha\in C(A)$, and therefore $\alpha\in A$.

    We claim that $A\subset B_\alpha(\pi/2)$. To see this, for any $x\in K$ and $t \in [0,1]$, consider $x_t=(1-t)\beta+tx=\beta+t(x-\beta)\in K$. Since $\beta$ is the minimum point,  $\|x_t\|\geq\|\beta\|$, i.e.
    $$\lan \beta, \beta\ran\leq\lan x_t, x_t\ran=\lan \beta+t(x-\beta), \beta+t(x-\beta)\ran=\lan \beta, \beta\ran+2t\lan \beta, x-\beta\ran+t^2\lan x-\beta, x-\beta\ran.$$
    Here $\lan\cdot,\cdot\ran$ is the standard inner product in $\R^3$.
    Thus $2t\lan \beta, x-\beta\ran+t^2\lan x-\beta, x-\beta\ran\geq0$ for all $t\in[0,1]$. This implies $\lan \beta, x-\beta\ran\geq0$, i.e., $\lan \beta, x\ran\geq\lan \beta, \beta\ran>0$. Hence $\lan \alpha, x\ran>0$, for all $x\in K$. By $A\subset K$, we have $\lan \alpha, x\ran>0$, for all $x\in A$, and therefore $A\subset B_\alpha(\pi/2)$.
\end{proof}

Suppose $X$ is a 3-dimensional convex set in $\H^3$ and $p \in \partial X$ in $\H^3$, a non-zero tangent vector $v$ in $T_p\H^3$ is called a \emph{normal vector} to $X$ or $\partial X$ at $p$, if the plane perpendicular to $v$ at $p$ is a supporting plane of $X$ at $p$. A normal vector $v$ is an \emph{exterior normal vector} if 
$v$ points away from $X$.   A tangent vector $w \in T_p \H^3$ is said to  \it point toward the interior of $X$ (or toward $X$) \rm if there exists a point $q \in \mathring X$ (or $q \in X$) such that $w$ is tangent to the geodesic segment from $p$ to $q$. It is well known that if a non-zero vector $w$ is the limit of a sequence of normal vectors to a convex set $X$ (at possibly different base points), then $w$ is a normal vector to $X$.

By a \emph{convex cap}, we mean a compact subsurface $K$ with boundary contained in a complete convex surface $S$ such that (1) the boundary of $K$ is a convex curve contained in a plane $\Pi$ and $K$ is contained in one closed half-space bounded by $\Pi$, and (2) the shortest distance projection from $K$ to $\Pi$ is injective. In fact, the surface $K$ is the intersection of $S$ with the closed half-space bounded by $\Pi$. In particular, the shortest distance projection $\pi$ is a homeomorphism from $K$ to the convex set $\Pi \cap X$ where $X$ is the closed convex region bounded by $S$.  Indeed, by the invariance of domain theorem, $\pi(K)$ is a compact 2-dimensional submanifold of $\Pi$ whose boundary is equal to the boundary of the convex set $\Pi \cap X$. Therefore, $\pi(K) =\Pi \cap X$.

The following lemma contains basic information about supporting planes and convex caps.  Part (iv) below is an analogous result of Busemann~\cite[Theorem 11.9]{Busemann} in the hyperbolic 3-space.

\begin{lemma}\label{lem:injection}
    Suppose $X$ is a 3-dimensional closed convex set in $\H^3$ and $p \in \partial X$.Then the following statements hold.
\begin{enumerate}[label=(\roman*)]
    \item There exists a normal vector pointing toward the interior of $X$ at $p$. 
    
    \item  If $X$ has a supporting plane at $p$ intersecting $X$ only at $p$, then there exists a normal vector pointing toward the interior of $X$ such that the supporting plane perpendicular to the vector intersects $X$ only at $p$.
 
    \item  If the normal vector to a supporting plane at $p$ points toward the interior of $X$,  then the shortest distance projection from $\partial X$ to the plane is injective in a neighborhood of $p$.

    \item  Suppose $X$ has a supporting plane at $p$ which intersects $X$ only at $p$, then the  point $p$ has a compact neighborhood $U\subset\partial X$, which is a convex cap. 

    \item The compact convex cap  $U$ in $\partial X$ produced in (iv) has the convexity property that for any $x,y \in U$, every shortest distance path $\gamma$ from $x$ to $y$ in $\partial X$ is contained in $ U.$ 
\end{enumerate}
\end{lemma}

\begin{proof}  
We verify \emph{(i)}--\emph{(v)} in order.

\smallskip
\noindent\emph{(i)} Let $T\subset \mathbb S^2$ consist of all unit vectors at $p$ pointing toward some $x\in X$, and $A\subset \mathbb S^2$ be the set of all exterior unit normal vectors of $\partial X$ at $p$.
We know both $A$ and $T$ are convex sets in $\mathbb S^2$. Since $X$ is 3-dimensional, $A$ does not contain a pair of antipodal points since otherwise, there are two supporting planes of $X$ at $p$ with opposite orientations, which is contradictory to $\dim(X)=3$. By Lemma \ref{lem:convex-hemisphere}, there exists $\alpha \in A$ such that $A\subset B_\alpha (\pi/2)$, i.e., $\lan -\alpha ,\delta\ran <0$ for all $\delta\in A$. By the basic duality theorem between the tangent cone and normal cone (\cite{Alexandrov}), the closure  $\overline{T}$ of  $T$ consists of all $y \in \mathbb S^2$ such that $\lan y, x \ran\leq0$ for all $x \in A$. Therefore $-\alpha$ is contained in $(\overline{T})^\circ=\mathring T$, and thus $-\alpha$ points toward $\mathring X$.

\smallskip
\noindent\emph{(ii)}
By assumption, there exists a supporting plane at $p$ with unit exterior normal vector $\alpha'$ such that the plane intersects $X$ only at $p$.  Using $\alpha'$ and the unit exterior normal vector $\alpha$ constructed in \emph{(i)}, we produce the required normal vector of $\partial X$ at $p$ as follows. 
Let $\beta$ be the normal vector of $\partial X$ at $p$ given by $(1-\epsilon)\alpha+\epsilon\alpha'$, for some small $\epsilon>0$.  Since $-\alpha$ points toward $\mathring X$, we can pick $\epsilon$ sufficiently small, such that $-\beta$ also points toward $\mathring X$. Furthermore, for any $\eta\in T$, $\lan\alpha ,\eta\ran\leq0$ since $\alpha \in A$. Hence, for all $\eta \in T$, we have
\begin{equation} \label{supP}
    \lan\beta,\eta\ran=\lan(1-\epsilon)\alpha+\epsilon\alpha',\eta\ran
    \leq\lan\epsilon\alpha',\eta\ran<0,
\end{equation}
where the last inequality follows from the condition that the supporting plane perpendicular to $\alpha'$ intersects $X$ only at $p$. The inequality \eqref{supP} then implies that the supporting plane perpendicular to $\beta$ intersects $\partial X$ in a single point $p$.

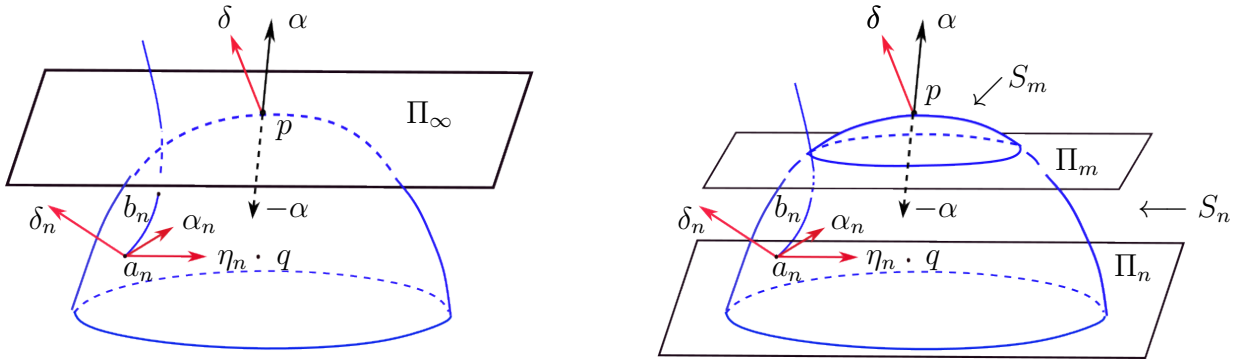
\begin{figure}[ht!]
\centering
\begin{overpic}[width=0.95\textwidth]{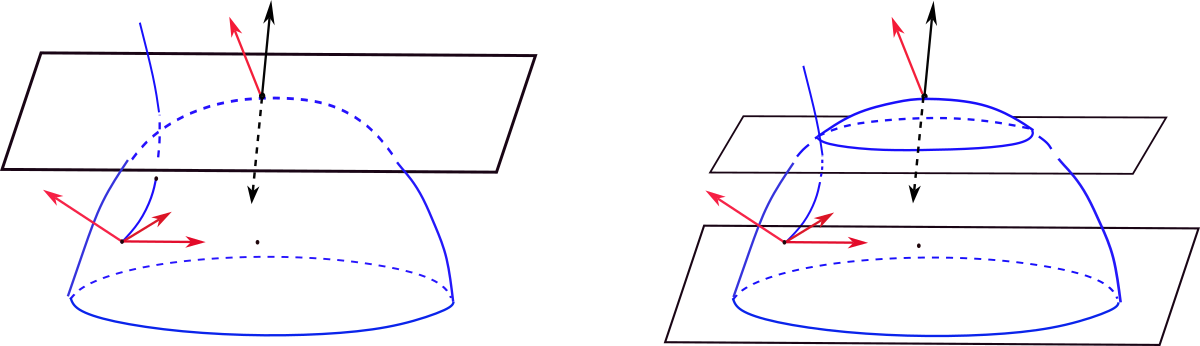}
    \put(18,28){$\delta$}
    \put(24,28){$\alpha$}
    \put(34,20){$\Pi_\infty$}
    \put(23,19){$p$}
    \put(22,12){$-\alpha$}
    \put(23,8){$q$}
    \put(18,8){$\eta_n$}
    \put(15,11){$\alpha_n$}
    \put(2,11){$\delta_n$}
    \put(10,7){$a_n$}
    \put(10,12){$b_n$}

     \put(73,28){$\delta$}
    \put(82,22){$\swarrow$}
    \put(85,23){$S_m$}

    \put(96,12){$\longleftarrow S_n$}

     \put(73,28){$\delta$}
    \put(79,28){$\alpha$}
    \put(89,16){$\Pi_m$}
    \put(94,7){$\Pi_n$}
    \put(65.2,12){$b_n$}

    \put(78,22){$p$}
    \put(77,12){$-\alpha$}
    \put(78,8){$q$}
    \put(73,8){$\eta_n$}
    \put(70,11){$\alpha_n$}
    \put(57,11){$\delta_n$}
    \put(65,7){$a_n$}
\end{overpic}
\caption{Pictures for part(iii) and part (iv).}
\label{5.3lemma1}
\end{figure}

\smallskip
\noindent\emph{(iii)}
Let $\Pi_{\infty}$ be the supporting plane to $S=\partial X$ at $p$ whose unit normal vector $-\alpha$ points toward $\mathring X$, $W^+_{\infty}$ be the closed half-space bounded by $\Pi_{\infty}$ disjoint from $\mathring X$,  and $\pi_{\infty}$ be the shortest-distance projection from $\H^3$ to $\Pi_{\infty}$. See the left figure in Figure \ref{5.3lemma1}.
    We will show that $\pi_{\infty}$ is injective on a neighborhood of $p$ in $S$.
    Suppose otherwise, there exist two infinite sequences of points $\{a_n\},\{b_n\}\subset S$ with $a_n,b_n\to p$ as $n\to\infty$, such that for all $n$,
    \begin{enumerate}[label=(\arabic*)]
        \item $a_n\neq b_n$;  and 
       
        \item  the oriented geodesic from $a_n$ to $b_n$             
        intersects $\Pi_{\infty}$ orthogonally and points into $W^+_{\infty}$, and in particular $\pi_{\infty}(a_n)=\pi_{\infty}(b_n)$.
    \end{enumerate}
    Let $\alpha_n$ be the unit vector at $a_n$ of the oriented geodesic from $a_n$ to $b_n$,  and $\delta_n$ be an exterior unit normal vector at $a_n$ to $S$. Since $b_n \in X$, 
    \begin{equation}\label{eq:inner-product}
        \lan \alpha_n,\delta_n\ran\leq 0.
    \end{equation}
    Since $\lim_{n\to\infty}a_n=p$, by passing to a subsequence, we may assume that $\{\delta_n\}$ converges to an exterior normal vector $\delta$ of $\partial X$ at $p$. 
    Since the geodesic $\gamma_n$ passing through $a_n$ and $b_n$ is orthogonal to $\Pi_{\infty}$, we know that the parallel transport of $\alpha_n$ along $\gamma_n$ to $\gamma_n\cap\Pi_{\infty}$ is a normal vector to $\Pi_{\infty}$. By condition (2) above and $\lim_{n\to\infty}a_n=p$, $\{\alpha_n\}$ converges to a normal vector of $S$ at $p$ perpendicular to $\Pi_{\infty}$ and pointing toward $W^+_{\infty}$. Hence $\lim_{n\to\infty}\alpha_n=\alpha$. By \eqref{eq:inner-product}, we have
    \begin{equation}\label{eq:inner-product3}
        \lan \alpha,\delta\ran\leq 0.
    \end{equation}

    On the other hand, let $q\in \mathring X$, such that $-\alpha$ is the unit tangent vector of the geodesic from $p$ to $q$. Consider the unit tangent vector $\eta_n$ at $a_n$ of the geodesic from $a_n$ to $q$. Then
    \begin{equation}\label{eq:inner-product2}
        \lan \eta_n,\delta_n\ran\leq 0.
    \end{equation}
    By $\lim_{n\to\infty}a_n=p$, $\{\eta_n\}$ converges to the unit vector at $p$ point toward $q$, and therefore $\lim_{n\to\infty}\eta_n=-\alpha$. By \eqref{eq:inner-product2}, we have
    \begin{equation}\label{eq:inner-product4}
        \lan -\alpha,\delta\ran\leq 0.
    \end{equation}

    Then by \eqref{eq:inner-product3}, \eqref{eq:inner-product4}, we have $\lan -\alpha,\delta\ran=0$. Since $\delta$ is an exterior unit normal vector to $S$ at $p$, we see $-\alpha$ and hence the geodesic $[p, q]$ are contained in the supporting plane of $S$ with normal vector $\delta$.  But this contradicts  $q\in \mathring X$.

\smallskip
\noindent\emph{(iv)}
By \emph{(ii)}, we can take $\alpha$ to be a unit exterior normal vector to the supporting plane $\Pi_{\infty}$ to $X$ at $p$ such that $-\alpha$ points into $\mathring X$ and $\Pi_{\infty} \cap X =\{p\}$. See the right figure in Figure \ref{5.3lemma1}. Let $q \in \mathring X$ be a point such that $-\alpha$ is tangent to the geodesic segment $[p,q]$ and  $q_n \in [p,q]$ be the point such that the distance $d(p,q_n)=d(p, q)/n$. The hyperbolic plane through $q_n$ perpendicular to $[p,q]$ is denoted by $\Pi_n$,  and the two closed half-spaces bounded by $\Pi_n$ are denoted by $W^{\pm}_n$ such that $p \in W^+_n$. Let the shortest distance projection from $\H^3$ to $\Pi_n$ be $\pi_n$. By construction, we see that the planes $\Pi_n$ converge in Hausdorff distance to $\Pi_{\infty}$,  $\lim_{n}\pi_n =\pi_{\infty}$, and $W^+_n \cap X$ converge in Hausdorff distance to $W^+_{\infty} \cap X =\Pi_{\infty} \cap X =\{p\}$.  This shows that $W^+_n \cap X$ is a compact convex set for $n$ large. On the other hand, since $q \in \mathring X$, $W^+_n \cap X$ is 3-dimensional. It follows that $\partial (W^+_n \cap X)$ is a closed convex surface which is a union of $S_n \cup (\Pi_n \cap X)$ where $S_n =\partial X \cap W^+_n$. 
Clearly $S_n$ is a closed neighborhood of $p$ in $\partial X$.
Furthermore, due to Hausdorff convergence of $W^+_n \cap X$ to $\{p\}$, the diameters of $W^+_n\cap X$ and $S_n$ tend to zero. 

We claim that for $n$ large, and any $m \geq n$, the shortest distance projection $\pi_m: S_n \to \Pi_m$ is injective. In particular, it shows $S_n$ is a convex cap when we take $m=n$. The proof follows exactly the same argument that we used in (iii). All we need to do is replace $\pi_{\infty}$ by $\pi_m$, $\Pi_{\infty}$ by $\Pi_m$, $S$ by $S_n$,  $W^+_{\infty}$ by $W^+_m$, $a_n, b_n \in S$ by $a_n, b_n \in S_n$ such that $\pi_{m_n}(a_n)=\pi_{m_n}(b_n)$ for some $m_n$ ($\geq n$) depending on $n$.  The key point in the proof of part (iii) is that $\lim_n a_n=\lim_n b_n =p$.  This follows since $W_n^+\cap X$ converges to $\{p\}$ in the Hausdorff sense. We omit the details.

\smallskip
\noindent\emph{(v)}
Using part \emph{(iv)}, let $N$ be an integer such that for all $m \geq N$, the shortest-distance projection $\pi_m: S_N \to \Pi_m$ is injective.  Take $m$ large such that the diameter of $S_m$ is smaller than the distance between $\Pi_m$ and $\Pi_N$. Let $Y_m =\pi_m^{-1}(X \cap \Pi_m)$ which is the union of all geodesics perpendicular to $\Pi_m$ at points in $\Pi_m \cap X$. $Y_m$ is a 3-dimensional closed convex set in $\H^3$ due to the convexity of $\Pi_m \cap X$.  Since $S_m$ is a convex cap, by the invariance of domain theorem mentioned before, $\pi_m(S_m) =X \cap \Pi_m$. By the injectivity of $\pi_m|_{S_N}$, we see that $\pi_m(S_N -S_m) \cap \pi_m(S_m) =\pi_m(S_N-S_m) \cap (X \cap \Pi_m) =\emptyset.$ This shows that $S_N-S_m$ is disjoint from $Y_m$. See Figure \ref{5.3lemma2}.

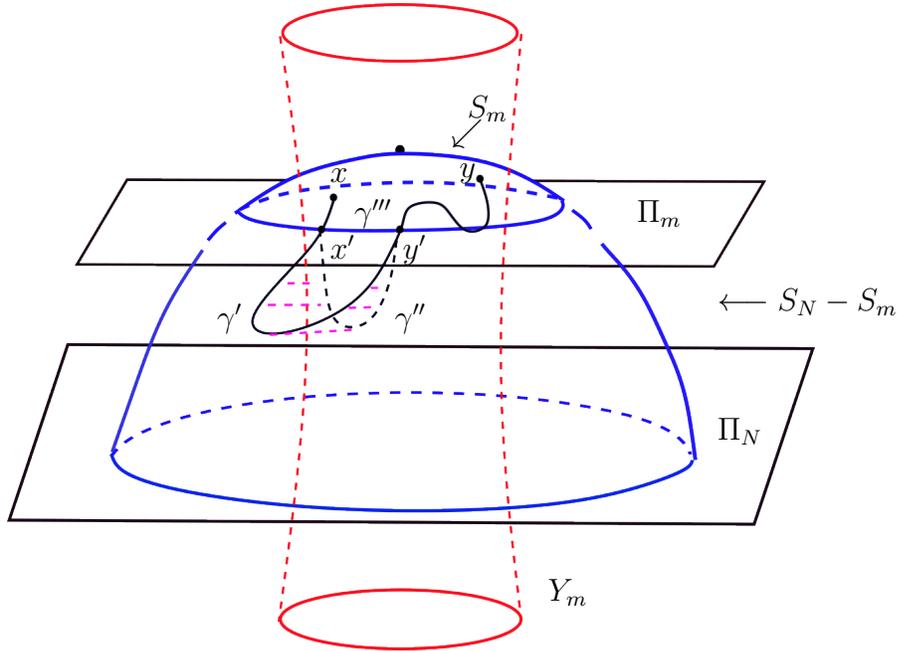
\begin{figure}[h!]
  \centering  
\begin{overpic}[width=0.65\textwidth]{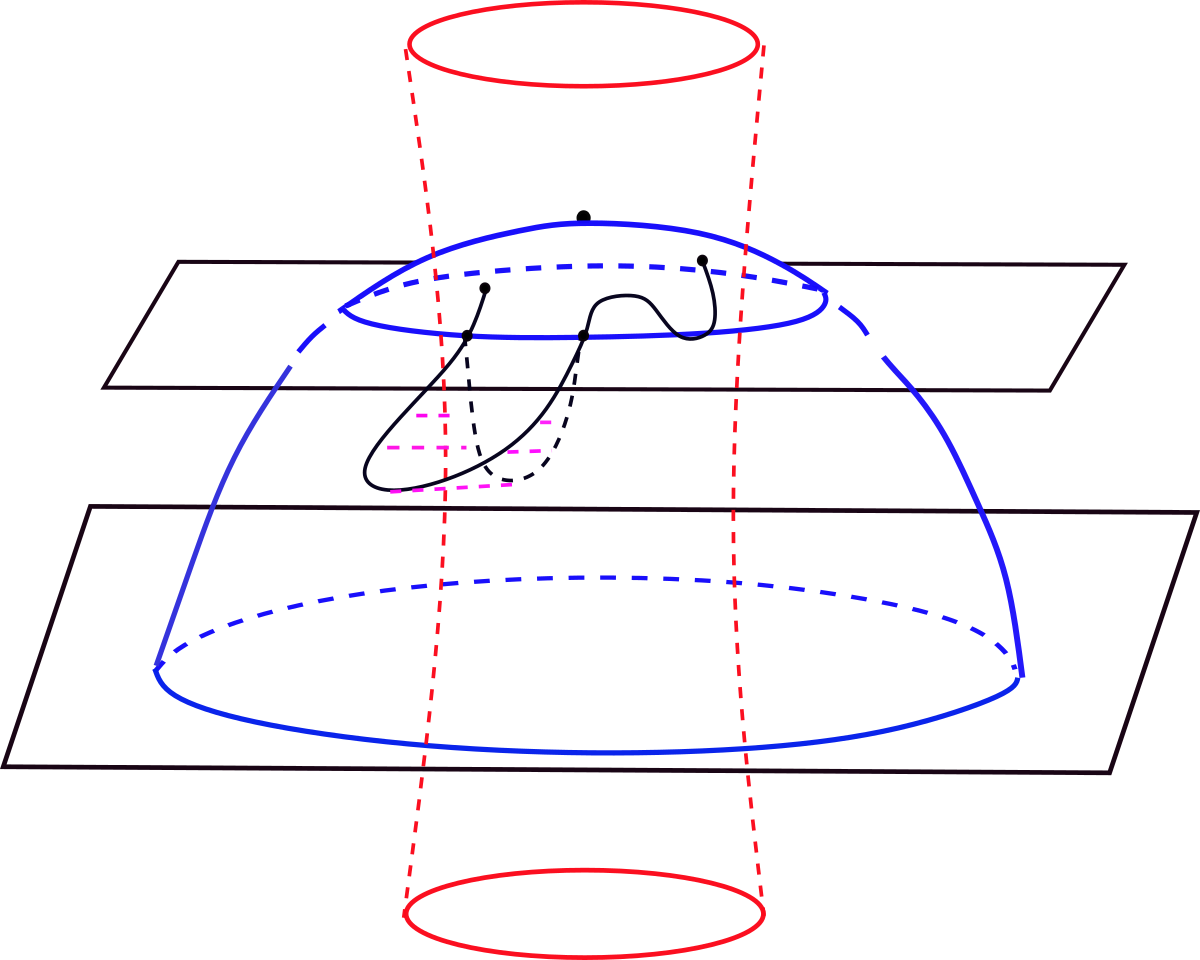}
    \put(43,53){$\gamma'''$}
    \put(40,48.5){$x'$}
    \put(49,48.5){$y'$}
    \put(26,40){$\gamma'$}
    \put(48,40){$\gamma''$}
    \put(40,58){$x$}
    \put(56,58.5){$y$}
    \put(78,53){$\Pi_m$}

    \put(88,26){$\Pi_N$}
    \put(88,42){$\longleftarrow S_N-S_m$}

    \put(55,63){$\swarrow$}
    \put(57,66){$S_m$}
    \put(67,6){$Y_m$}

\end{overpic}
\caption{Picture for part (v).}
\label{5.3lemma2}
\end{figure}

Now we prove that $S_m$ is convex in the sense that for any $x, y \in S_m$, each shortest path $\gamma$ in $\partial X$ from $x$ to $y$ is contained in $S_m$. First of all, since the diameter of $S_m$ is smaller than the distance between $\Pi_m$ and $\Pi_N$, we see that $\gamma \subset W^+_N$, and hence $\gamma \subset S_N$. If $\gamma$ is not contained in $S_m$, there exists a subarc $\gamma'$ of $\gamma$ joining $x'$ to $y'$ in $\partial S_m$ such that $\gamma' -\{x', y'\} \subset S_N-S_m$.  Let $\pi_Y$ be the shortest distance projection from $\H^3$ to $Y_m$. 
Since $\gamma' -\{x', y'\}$ is contained $S_N-S_m$, we see that $\gamma'$ is disjoint from the interior of $Y_m$. It follows that the shortest-distance projection $\gamma'' =\pi_Y(\gamma')$ of $\gamma'$ to $\partial Y_m$ is a path from $x'$ to $y'$ whose length is at most that of $\gamma'$.  The decrease in length is due to the fundamental fact that the shortest-distance projection is distance-decreasing, i.e., 1-Lipschitz. See, for instance, \cite[Lemma 1.3.4]{EM87}, or \cite[Theorem 11.2]{Busemann}. 
Consider the path $\gamma'''=\pi_m(\gamma'') \subset \partial S_m$ which joins $x'$ to $y'$.  By the 1-Lipschitz property of $\pi_m$, we see that $l(\gamma''') \leq l(\gamma'')$. However, due to $\gamma' -\{x, y\} \subset S_N -S_m$ and $\pi^{-1}_Y(\partial S_m) \subset \Pi_m$, $\gamma''$ is not contained in $\Pi_m$. Therefore $l(\gamma''') < l(\gamma'')$.  In summary, we have $l(\gamma''') < l(\gamma'') \leq l(\gamma')$. 
Replacing $\gamma'$ portion of $\gamma$ by $\gamma'''$, we obtain a path in $\partial X$ from $x$ to $y$ whose length is strictly less than the length of the shortest path $\gamma$. This contradiction shows $\gamma \subset S_m.$
\end{proof}

Now we finish the proof of Theorem \ref{main2} for case (iii).
Suppose otherwise, 
without loss of generality, we assume that $\dim(X_1)=2$ and $\dim(X_2)=3$. 
Let $D(X_1)$ be the double of $X_1$ as in Section \ref{subsec:2d-2d}, where $Y_1$ is the boundary of $X_1$ as a 2-dimensional convex set in $\H^2$, and  $f\from D(X_1)\to \partial X_2$ be an isometry.

By construction, there exists an isometric reflection $\sigma \from D(X_1)\to D(X_1)$ that maps $\mathring{X}_1^\pm$ to $\mathring{X}^{\mp}_1$ and has $Y_1$ as the fixed point set.
Then $g=f\circ  \sigma \circ f^{-1}\from \partial X_2\to\partial X_2$ is an isometry reflection whose fixed point set is $f(Y_1)$.
By our result on Theorem \ref{main2} case (i), $g$ is the restriction of some $\phi\in\Isom(\H^3)$ to $\partial X_2$. Since $g \circ g=id$, we see that $\phi \circ \phi$ fixes every point in $\partial X_2$. But $\partial X_2$ is not contained in a plane. It follows that $\phi \circ \phi$ is the identity map. 
Since $Y_1$ has infinitely many points, so does $f(Y_1)$. Then the isometric involution $\phi$ has infinitely many points. Hence, it can only be a hyperbolic reflection about a plane or a $\pi$-rotation about a geodesic $L$.  We claim that $\phi$ cannot be a $\pi$-rotation. Suppose otherwise that $\phi$ is a $\pi$-rotation, 
since $f(Y_1) \subset L \cap \partial X_2$ and $f(Y_1)$ has infinitely many points, we can take three different points in $f(Y_1)\subset L$. Considering the supporting plane $\Pi'$ of $X_2$ at the point that lies between the other two points on the geodesic $L$, we see that the remaining two points cannot lie on the opposite sides of $\Pi'$, and hence they are all contained in $\Pi'$. This shows that the supporting plane $\Pi'$ contains $L$ and $\Pi'$ is invariant under $\phi$. In particular, $\phi$ interchanges the two sides of the supporting plane. As a consequence, $X_2$ is contained in the supporting plane. But this contradicts $\dim(X_2)=3$.

Therefore, $\phi$ must be a hyperbolic reflection about a hyperbolic plane $P $. Let $\partial X_2^\pm=f(X_1^\pm)$ be the two parts of $\partial X_2$ separated by $P$.
Since $\dim(X_2)=3$, there exists a closed half-space $Q_1$ with boundary plane $\Pi_1$ such that the hyperbolic distance between $Q_1$ and $P $ is positive and $S=Q_1 \cap \partial X_2^+$ is a non-empty surface with boundary. Because all the ideal points $\overline{X_1} \cap \partial \H^3$ of $X_1$ are contained in the conformal boundary at infinity of the plane containing $X_1$,  the ideal points $\overline{X_2} \cap \partial \H^3$ of $X_2$ are contained in the boundary of $P $ at infinity. It follows that $S$ is compact by the choice of $Q_1$. Furthermore, $S$ is a convex surface since it is contained in the boundary of the convex set $X_2 \cap Q_1$. The compactness of $S$ implies there exists $p\in S$ such that
\begin{equation}
    d(p,\Pi_1)=\max\{d(x,\Pi_1):x\in S\}:=\kappa >0. 
    \end{equation}
Consider the equidistant surface $\tilde{S}$ passing through $p$ that consists of all points $x$ of distance $\kappa$ to $\Pi_1$. It is well known that $\tilde{S}$ is a strictly convex surface. Let $K$ be the closed convex set bounded by  $\Pi_1\cup\tilde{S}$. By the maximality of $d(p, \Pi_1)$, we have $S\subset K$. By strict convexity, the supporting plane of $\tilde{S}$ at $p$ intersects $\tilde S$ only at $p$. Therefore, due to $S \subset K$, the same plane is a supporting plane for $S$ and intersects $S$ only at $p$.  By Lemma \ref{lem:injection} (iv) and (v) applied to $\partial X_2$, $p$ has a compact neighborhood $U$ which is a convex cap and is convex with respect to the intrinsic path metric of ${\partial X_2}$.

Let $\Pi$ be the plane containing $\partial U$,  and $W$ be the compact convex set bounded by  $U$ and $\Pi$. 
Let $\tilde{W}$ be the union of $W$ and the image of $W$ under the hyperbolic reflection about $\Pi$.  We claim that $\tilde{W}$ is a convex set in $\H^3$. Indeed, for any $a,b\in\tilde{W}$, if they lie in the same side of $\Pi$, then the geodesic segment $[a,b]$ joining $a$ and $b$ is contained in $\tilde{W}$ since $W$ is convex. If $a,b$ lie on opposite sides of $\Pi$, we project $a,b$ to $\Pi$ by the short-distance projection $\pi$ and obtain $a'=\pi(a),b'=\pi(b)$, respectively. Since $\pi(U)=\Pi\cap X_2$ is convex, the geodesic segment $[a',b']$ joining $a'$ and $b'$ is contained in $W \subset \tilde{W}$. Since  $a'=\pi(a)$ and $b'=\pi(b)$, $a,b,a', b'$ are contained in a hyperbolic plane $\Pi_2$ which is the preimage under $\pi$ of a geodesic containing $a', b'$. It follows that two geodesic segments $[a,b]$ and $[a', b']$ must intersect in a point $c$ in the plane $\Pi_2$. 
Without loss of generality, we assume $a\in W$ and $b$ is in the image of the reflection of $W$ across  $\Pi$. By the convexity of $W$ and $c\in W$, we then know $[a,c], [b,c] \subset \tilde{W}$, and hence $[a,b] \subset \tilde{W}$. This shows that $\tilde{W}$ is a convex set, and  $\partial\tilde{W}$ is a closed convex surface. By construction, $\partial\tilde{W}$ is the metric double of $U$ across $\partial U$.

Let $V=f^{-1}(U)\subset X_1^+$. Consider the metric double $D(V)$ of $V$ along its boundary. By Lemma \ref{lem:injection} (iv), $U$ is convex with respect to the intrinsic path metric of $X_1^+$. Since  $X_1^+$ is isometric to a convex set in $\H^2$ and $V$ is isometric to $U$, we see  $V$ is isometric to a 2-dimensional closed convex set in $\H^2$. In particular,  the metric double surface $D(V)$ is a degenerate closed convex surface as we discussed in \S5.2. Furthermore, by construction,  $D(V)$ is isometric to $\partial\tilde{W}$, which is the boundary of a compact convex body $\tilde W$ in $\H^3$. By Pogorelov's theorem \ref{pogo1} applied to $D(V)$ and $\partial \tilde W$, we obtain a contradiction on the dimensions of $V$ and $\tilde W$. This ends the proof of case (iii).

\bibliography{reference2}
\bibliographystyle{alpha}

\end{document}